\documentclass[12pt,british,a4paper]{amsart}

\usepackage{babel,amssymb,url,stmaryrd,mathtools,fixltx2e}
\usepackage{paralist,csquotes,stackrel}

\numberwithin{equation}{section}
\mathtoolsset{mathic=true}
\setdefaultenum{\upshape (i)}{}{}{}

\newcommand{\ii}{\mathbf i}
\newcommand{\jj}{\mathbf j}

\newcommand{\cO}{\mathcal O}
\newcommand{\tH}{\tilde H}

\newcommand{\tmu}{\tilde\mu}

\theoremstyle{plain}
\newtheorem{theorem}[equation]{Theorem}
\newtheorem{proposition}[equation]{Proposition}
\newtheorem{lemma}[equation]{Lemma}

\theoremstyle{definition}
\newtheorem{definition}[equation]{Definition}

\theoremstyle{remark}
\newtheorem{Example}[equation]{Example}
\newtheorem{remark}[equation]{Remark}

\newenvironment{example}{\begin{Example}\pushQED{\qee}}{\popQED\end{Example}}
\DeclareRobustCommand{\qee}{%
  \ifmmode \mathqee
  \else
    \leavevmode\unskip\penalty9999 \hbox{}\nobreak\hfill
    \quad\hbox{\qeesymbol}%
  \fi
}
\newcommand{\mathqee}{\quad\hbox{\qeesymbol}}
\newcommand{\qeesymbol}{\ensuremath\diamondsuit}

\newcommand{\Lie}[1]{\textsl{#1}}
\newcommand{\lie}[1]{\mathfrak{#1}}
\newcommand{\GL}{\Lie{GL}}

\newcommand{\sL}{\lie{sl}}

\newcommand{\su}{\lie{su}}

\newcommand{\kf}{\lie k}

\newcommand{\q}{\lie q}
\newcommand{\tf}{\lie t}

\newcommand{\SL}{\Lie{SL}}
\newcommand{\SU}{\Lie{SU}}
\newcommand{\Un}{\Lie{U}}

\newcommand{\twist}{\mathcal{Z}}

\newcommand\C{{\mathbb C}}
\newcommand{\HH}{{\mathbb H}}
\newcommand{\PP}{{\mathbb P}}
\newcommand{\R}{{\mathbb R}}

\newcommand{\liek}{\lie k}
\newcommand{\hkimpl}{{\mathrm{hkimpl}}}
\newcommand{\impl}{{\mathrm{impl}}}
\newcommand{\hks}{{\mathrm{hks}}}

\newcommand{\symp}{{\sslash}} 

\newcommand{\hkq}{{\sslash\mkern-6mu/}}

\DeclareMathOperator{\Hom}{Hom}

\DeclareMathOperator{\LIE}{Lie}
\DeclareMathOperator{\Nil}{Nil}

\DeclareMathOperator{\Spec}{Spec}

\DeclareMathOperator{\tr}{tr}

\DeclarePairedDelimiter{\abs}{\lvert}{\rvert}

\newcommand{\eqbreak}[1][2]{\\&\hskip#1em}

\begin{document}

\title{Twistor spaces for hyperk\"ahler implosions}

\author{Andrew Dancer}
\address[Dancer]{Jesus College\\
Oxford\\
OX1 3DW\\
United Kingdom\\} \email{dancer@maths.ox.ac.uk}

\author{Frances Kirwan}
\address[Kirwan]{Balliol College\\
Oxford\\
OX1 3BJ\\
United Kingdom\\} \email{kirwan@maths.ox.ac.uk}

\author{Andrew Swann}
\address[Swann]{Department of Mathematics\\
Aarhus University\\
Ny Munkegade 118, Bldg 1530\\
DK-8000 Aarhus C\\
Denmark\\
\textit{and}\\
CP\textsuperscript3-Origins,
Centre of Excellence for Cosmology and Particle Physics Phenomenology\\
University of Southern Denmark\\
Campusvej 55\\
DK-5230 Odense M\\
Denmark\\} \email{swann@imf.au.dk}
\dedicatory{\large Dedicated to the memory of Friedrich Hirzebruch}

\begin{abstract}
  We study the geometry of the twistor space of the universal
  hyperk\"ahler implosion \( Q \) for \( \SU(n) \).  Using the
  description of \( Q \) as a hyperk\"ahler quiver variety, we
  construct a holomorphic map from the twistor space \( \mathcal Z_Q
  \) of \( Q \) to a complex vector bundle over \( \mathbb P^1 \), and
  an associated map of \( Q \) to the affine space \( \mathcal R \) of
  the bundle's holomorphic sections.  The map from \( Q \) to \(
  \mathcal R \) is shown to be injective and equivariant for the
  action of \( \SU(n) \times T^{n-1} \times \SU(2) \).  Both maps,
  from \( Q \) and from \( \mathcal Z_Q \), are described in detail
  for \( n=2 \) and \( n=3 \).  We explain how the maps are built from
  the fundamental irreducible representations of \( \SU(n) \) and the
  hypertoric variety associated to the hyperplane arrangement given by
  the root planes in the Lie algebra of the maximal torus.  This
  indicates that the constructions might extend to universal
  hyperk\"ahler implosions for other compact groups.
\end{abstract}


\maketitle

\setcounter{section}{-1}

\section{Introduction}
\label{sec:introduction}
 
In \cite{DKS,DKS2} we introduced a hyperk\"ahler analogue in the case
of \( \SU(n) \) actions of Guillemin, Jeffrey and Sjamaar's
construction of \emph{symplectic implosion}~\cite{GJS}. The aim of this
paper is to find a description of hyperk\"ahler 
implosion for \( \SU(n) \) which can
be generalised to other compact groups \( K \).

If \( M \) is a symplectic manifold with a Hamiltonian action of a
compact group \( K \) with maximal torus \( T \), the symplectic
implosion \( M_\impl \) is a stratified symplectic space (usually
singular) with an action of \( T \) such that the symplectic
reductions of \( M_\impl \) by \( T \) coincide with the symplectic
reductions of \( M \) by \( K \).  The implosion of the cotangent
bundle \( T^*K \cong K \times \lie k^*\) acts as a universal
symplectic implosion in that the implosion of a general Hamiltonian \(
K \)-manifold \( M \) can be identified with the symplectic reduction
by \( K \) of \( M \times (T^*K)_\impl \).

The construction in \cite{DKS} of the universal hyperk\"ahler
implosion when \( K=\SU(n) \) uses quiver diagrams, and gives us a
stratified hyperk\"ahler space \( Q =(T^*K_\C)_{\hkimpl}\).  The
hyperk\"ahler implosion of a general hyperk\"ahler manifold \( M \)
with a Hamiltonian action of \( K=\SU(n) \) is then defined as the
hyperk\"ahler reduction by \( K \) of \( M \times Q \).  The
hyperk\"ahler strata of \( Q \) can be described in terms of open sets
in complex symplectic quotients of the cotangent bundle of \(
K_\C=\SL(n, \C) \) by subgroups which are extensions of abelian groups
by commutators of parabolic subgroups.  There is an action of the
maximal torus \( T \) of \( K \), and the hyperk\"ahler quotients by
this action are the Kostant varieties, affine varieties which are
closures in \( \lie k_\C^* \) of complex coadjoint orbits.  As in
\cite{DKS} we will identify Lie algebras with their duals via an
invariant inner product, so that coadjoint orbits of \( K_\C \) are
identified with adjoint orbits. 

The universal symplectic implosion has a natural \( K \times T
\)-equivariant embedding into a complex affine space, whose image is
the \( K \)-sweep \( K \overline{T_\C v} \) of the closure of an orbit
\( T_\C v \) of the complexified maximal torus \( T_\C \)
\cite{GJS}. Here \( \overline{T_\C v} \) is the toric variety
associated to a positive Weyl chamber \( \lie{t}_+ \) in the Lie
algebra \( \lie{t} \) of \( T \). It was shown in \cite{DKS2} that the
hypertoric variety associated to the hyperplane arrangement given by
the root planes in \( \lie{t} \) maps generically injectively to \( Q
\) and that (for any choice of complex structure on \( Q \)) the \(
K_\C \)-sweep \( K_\C Q_T \) of its image \( Q_T \) is dense in \( Q
\). In this paper we shall construct a \( K \times T \)-equivariant
embedding \( \sigma \) of the universal hyperk\"ahler implosion \( Q
\) for \( K=\SU(n) \) into a complex affine space \( \mathcal{R} \)
with a natural \( \SU(2) \)-action which rotates the complex
structures on \( Q =(T^*K_\C)_{\hkimpl}\). This embedding is
constructed using the moment maps for the \( K \times T \) action and
the fundamental irreducible representations of \( K \). Its image is
the closure of the \( K_\C \)-sweep of the image \( \sigma(Q_T) \) in
\( \mathcal{R} \) of the hypertoric variety associated to the
hyperplane arrangement given by the root planes in \( \lie{t} \).  Our
future aim is to use this description of \( Q \) as the closure of the
\( K_\C \)-sweep of the image of this hypertoric variety in the \( K
\times T \times \SU(2) \)-representation \( \mathcal{R} \) to extend
the hyperk\"ahler implosion construction from \( K=\SU(n) \) to more
general compact groups \( K \).
 
Any hyperk\"ahler manifold \( M \) has a twistor space \( \twist_M \),
which is a complex manifold with additional structure from which we
can recover \( M \). As a smooth manifold \( \twist_M \) is the
product \( M \times \PP^1 \) of \( M \) and the complex projective
line \( \PP^1 \), which is identified with the unit sphere \( S^2 \)
in \( \R^3 \) in the usual way. The complex structure on \( \twist_M
\) is such that the projection \( \pi\colon \twist_M \to \PP^1 \) is
holomorphic and its fibre at any \( \zeta \in \PP^1 = S^2 \) is \( M
\) equipped with the complex structure determined by \( \zeta \).  

We shall give a description of the twistor space for the universal
hyperk\"ahler implosion for \( \SU(n) \). Motivated by the embedding
of the implosion into the affine space \( \mathcal{R} \) described above,
we shall also construct a generically injective holomorphic map from
its twistor space to a vector bundle over \( \PP^1 \).
 
The layout of the paper is as follows. In \S \ref {sec:twistor-spaces} we
review the theory of twistor spaces for hyperk\"ahler manifolds, and
in \S \ref {sec:nilpotent-cone} we recall the hyperk\"ahler structure on
the nilpotent cone in the Lie algebra of \( K_\C=\SL(n, \C) \)
obtained in \cite{KS}. We also describe its twistor space which can be
embedded in the vector bundle \( \cO(2) \otimes \lie k_\C \) over \(
\PP^1 \), where \( \lie k_\C \) is the Lie algebra of the
complexification \( K_\C=\SL(n, \C) \) of \( K = \SU(n) \). In
\S \ref {sec:symplectic-implosion} we recall the constructions of
symplectic implosion from \cite{GJS} and hyperk\"ahler implosion for
\( K=\SU(n) \) from \cite{DKS,DKS2}. In
\S \ref {sec:towards-an-embedding} we define a \( K \times T \times
\SU(2) \)-equivariant map \( \sigma \) from the universal
hyperk\"ahler implosion \( Q \) for \( K=\SU(n) \) to
\begin{equation*}
  \mathcal{R} = 
  H^0(\PP^1,
  (\cO(2) \otimes (\lie k_\C \oplus \lie t_\C)) \oplus 
  \bigoplus_{j=1}^{n-1}  \cO(\ell_j) \otimes \wedge^j\C^n)
\end{equation*}
where $\ell_j = j(n-j)$,
and an associated holomorphic map \( \tilde{\sigma} \) from the
twistor space \( \twist_Q \) of \( Q \) to the vector bundle \(
(\cO(2) \otimes (\lie k_\C \oplus \lie t_\C)) \oplus
\bigoplus_{j=1}^{n-1} \cO(\ell_j) \otimes \wedge^j\C^n \) over \( \PP^1
\). Here \( \SU(2) \) acts on \( \mathcal{R} \) via its usual action
on \( \PP^1 \) and the line bundles \( \cO(\ell_j) \) over \( \PP^1 \) for
\( j \in \mathbb{Z} \), while \( K = \SU(n) \) acts on \( \mathcal{R}
\) via its adjoint action on \( \liek_\C = \liek \otimes_\R \C \), the
trivial action on \( \lie{t}_\C \) and its usual action on \( \wedge^j
\C^n \). The action of \( T \) on \( \cO(2) \otimes (\lie k_\C \oplus
\lie t_\C) \) is the restriction of the \( K \)-action, but \( T \)
acts on \( \cO(\ell_j) \otimes \wedge^j\C^n \) as multiplication by the
highest weight for the irreducible representation \( \wedge^j\C^n \)
of \( K= \SU(n) \).

In \S \ref {sec:strat-univ-hyperk} we recall the stratification given in
\cite{DKS} of \( Q \) into strata which are hyperk\"ahler manifolds,
and its refinement in \cite{DKS2} which has strata \( Q_{[\sim,\cO]}
\) indexed in terms of Levi subgroups and nilpotent orbits in \(
K_\C=\SL(n, \C) \).  In \S \ref {sec:embeddings} we prove that the map \(
\sigma \) defined in \S \ref {sec:towards-an-embedding} is injective and
that \( \tilde{\sigma} \) is generically injective; we will see that
\( \tilde{\sigma} \) fails to be injective in the example \( n=2 \) in
\S \ref {sec:twist-space-univ}. In \S \ref {sec:twist-space-univ} we
describe the full structure of the twistor space of \( Q \) in terms
of its embedding in the space of holomorphic sections of the vector
bundle \( (\cO(2) \otimes (\lie k_\C \oplus \lie t_\C)) \oplus
\bigoplus_{j=1}^{n-1} \cO(\ell_j) \otimes \wedge^j\C^n \) over \( \PP^1 \),
and we consider the low-dimensional examples \( n=2 \) and \( n=3 \) in
detail. Finally in \S \ref {sec:general-compact-K} we consider how to use
the description of \( Q \) as the closure of the \( K_\C \)-sweep of the image
of a hypertoric variety in the \( K \times T \times
\SU(2) \)-representation \( \mathcal{R} \) and the corresponding description
of its twistor space \( \twist_Q \) to extend the hyperk\"ahler implosion
construction from \( K=\SU(n) \) to more general compact groups \( K \).

\subsubsection*{Acknowledgements.} The work of the second author was
supported by a Senior Research Fellowship of the Engineering and
Physical Sciences Research Council (grant number GR/T016170/1) during
much of this project. The third author is partially supported by the
Danish Council for Independent Research, Natural Sciences.

\section{Twistor spaces}
\label{sec:twistor-spaces}

In this section we review the theory of twistor spaces for
hyperk\"ahler manifolds; for more details see \cite{HKLR}.

A hyperk\"ahler manifold \( M \) has a Riemannian metric, together
with a triple of complex structures \( (\mathsf{I},\mathsf{J},\mathsf{K}) \)
satisfying the quaternionic relations such that the metric is K\"ahler
with respect to each complex structure. Thus a hyperk\"ahler manifold
has a triple (in fact a whole two-sphere) of symplectic forms \(
(\omega_1,\omega_2,\omega_3)\) which are K\"ahler forms for the
complex structures \( (\mathsf{I},\mathsf{J},\mathsf{K}) \).

If a compact group \( K \) acts on a hyperk\"ahler manifold \( M \)
preserving its hyperk\"ahler structure with moment maps \( \mu_1,
\mu_2, \mu_3 \) for the symplectic forms \( \omega_1, \omega_2,
\omega_3\), then the hyperk\"ahler quotient \( \mu^{-1}(0)/K \) (where
\( \mu = (\mu_1,\mu_2, \mu_3) \colon\allowbreak M \to \lie k \otimes
\R^3 \)) inherits a hyperk\"ahler structure from that on \( M \).

A hyperk\"ahler manifold \( M \) has (real) dimension \( 4k \) for
some non-negative integer \( k \). We can associate to \( M \) its
twistor space \( \twist_M \) which is a complex manifold of (complex)
dimension \( 2k+1 \) with some additional structure from which we can
recover the hyperk\"ahler manifold \( M \). As a smooth manifold \(
\twist_M \) is the product \( M \times S^2 \) of \( M \) and the
two-dimensional sphere \( S^2 \), but its complex structure at \(
(m,\zeta) \in M \times S^2 \) is defined by \( ({\sf
I}_\zeta,\mathsf{\tilde I}) \) where \( \mathsf{\tilde I} \) is the
usual complex structure on \( S^2 \cong \PP^1 \) and if
\begin{equation*}
  \zeta=(\zeta_1,\zeta_2,\zeta_3) \in S^2 \subset \R^3
\end{equation*}
then \( \mathsf{I}_\zeta = \zeta_1 \mathsf{I} + \zeta_2 \mathsf{J} +
\zeta_3 \mathsf{K} \), where \( (\mathsf{I} ,\mathsf{J}, \mathsf{K})
\) is the triple of complex structures on \( M \) as above.  The
twistor space \( \twist_M \) is equipped with the following additional
structure \cite{HKLR}:
\begin{asparaenum}
\item a holomorphic projection \( \pi \colon \twist_M \to \PP^1 \)
  whose fibre at \( \zeta \in \PP^1 \) is \( M \) equipped with the
  holomorphic structure \( \mathsf{I}_\zeta \) determined by \( \zeta \);
\item a holomorphic section \( \omega \) of the holomorphic vector
  bundle
  \begin{equation*}
    \wedge^2 T_F^*(2) = \wedge^2 T_F^* \otimes \cO(2)
  \end{equation*}
  over \( \twist_M \) where \( T_F \) is the tangent bundle along the
  fibres of \( \pi \), such that \( \omega \) defines a holomorphic
  symplectic form \( \omega_\zeta \) on each fibre \( \pi^{-1}(\zeta)
  \cong M \) of \( \pi \);
\item a real structure (that is, an anti-holomorphic involution) \(
  \tau \) on \( \twist_M \) preserving this data and covering the
  antipodal map on \( \PP^1 \).
\end{asparaenum}
\medskip With respect to the \( C^\infty \)-identification of \(
\twist_M \) with \( M \times \PP^1 \) the real structure is given by
\begin{equation*}
  \tau(m,\zeta) = (m, -1/\bar{\zeta}).
\end{equation*}
With respect to the fixed holomorphic section \(
(1/2)\partial/\partial \zeta \) of \( T\PP^1 \cong \cO(2) \) the
holomorphic symplectic form \( \omega_\zeta \) is given by
\begin{equation*}
  \omega_\zeta = \omega_2 + \ii \omega_3 - 2 \zeta \omega_1 -
  \zeta^2(\omega_2 - \ii \omega_3), 
\end{equation*}
where \( \omega_1, \omega_2, \omega_3 \) are the K\"ahler forms
associated to the hyperk\"ahler metric and the complex structures \(
\mathsf{I}, \mathsf{J}, \mathsf{K} \). 

A holomorphic section of \( \pi \colon \twist_M \to \PP^1 \) is called
a twistor line, and a twistor line \( \sigma \colon \PP^1 \to \twist_M
\) is real if \( \tau\sigma(\zeta) = \sigma(-1/\bar{\zeta}) \) for
every \( \zeta \in \PP^1 \). Each point \( m \in M \) gives rise to a
real twistor line \( \{m\} \times \PP^1 \) with normal bundle \(
\C^{2k} \otimes \cO(1) \), and using such twistor lines we can recover
the hyperk\"ahler manifold \( M \) from its twistor space.

If a compact group \( K \) acts on a hyperk\"ahler
manifold \( M \) preserving its hyperk\"ahler structure with a
hyperk\"ahler moment map \( \mu= (\mu_1, \mu_2, \mu_3) \), then there
is an associated holomorphic map \( \twist_\mu \colon \twist_M \to
\lie k^*_\C \otimes \cO(2) \) whose restriction to each fibre of \(
\pi \) is a complex moment map for the holomorphic symplectic form
defined by \( \omega \) on the fibre. The twistor space of the
hyperk\"ahler quotient \( \mu^{-1}(0)/K \) is the quotient (in the
sense of K\"ahler geometry or geometric invariant theory
\cite{Kirwan:quotients}) of \( \twist_\mu^{-1}(0) \) by the
complexification \( K_\C \) of \( K \).

\begin{remark}
  \label{hnought}
  The \lq twistor moment map' \( \twist_\mu\colon \twist_M \to \lie
  k^*_\C \otimes \cO(2) \) restricts to a holomorphic section of \(
  \lie k^*_\C \otimes \cO(2) \) on each twistor line \( \{m\} \times
  \PP^1 \) in \( \twist_M \). This gives us a map
  \begin{equation*}
    \phi\colon M \to H^0(\PP^1, \lie k^*_\C \otimes \cO(2)) \cong \lie
    k^*_\C \otimes H^0(\PP^1,\cO(2))
  \end{equation*}
  whose evaluation at any \( p \in \PP^1 \) can be identified (modulo
  choosing a basis for the one-dimensional complex vector space \(
  \cO(2)_p \)) with the complex moment map associated to the
  corresponding complex structure on \( M \). Then \( \phi^{-1}(0) =
  \mu^{-1}(0) \) and the hyperk\"ahler quotient of \( M \) by \( K \)
  is \( \phi^{-1}(0)/K \).
\end{remark}
The twistor space for a flat hyperk\"ahler manifold \( \HH^k \) is the
vector bundle \( \twist = \cO(1) \otimes \C^{2k} \) over \( \PP^1 \).
If we write \( \twist \) as
\begin{equation*}
  \twist = \cO(1) \otimes (W \oplus W^*)
\end{equation*}
where \( W = \C^k \), then the natural pairing between \( W \) and \(
W^* \) defines a constant holomorphic section \( \omega \) of \(
\wedge^2 T_F^*(2) \). The standard hermitian structure on \( W \)
gives us an identification of \( W \) with \( \overline{W}^* \), and
this together with the antipodal map on \( \PP^1 \) gives us the real
structure on \( \twist \). 

Let \( \alpha_1,\ldots,\alpha_k \) be the standard coordinates on \( W
\) and let \( \beta_1, \ldots, \beta_k \) be the dual coordinates on
\( W^* \).  We can cover \( \twist \) with two coordinate patches \(
\zeta \neq \infty \) and \( \zeta \neq 0 \) with coordinates
\begin{equation} \label{labelcoord} (\alpha_1,\ldots,\alpha_k,\beta_1,
  \ldots, \beta_k,\zeta) \quad\text{and}\quad
  (\tilde{\alpha}_1,\ldots,\tilde{\alpha}_k,\tilde{\beta}_1, \ldots,
  \tilde{\beta}_k,\tilde{\zeta}) \end{equation} related by the
transition functions
\begin{equation*}
  \tilde{\zeta} = 1/\zeta, \,\, \tilde{\alpha}_j = \alpha_j/\zeta,
  \,\, \tilde{\beta}_j = \beta_j/\zeta.
\end{equation*}
With respect to these coordinates the real structure on \( \twist \)
is given by
\begin{equation*}
  (\alpha_1,\ldots,\alpha_k,\beta_1, \ldots, \beta_k,\zeta) \mapsto 
  (\bar{\beta}_1/\bar{\zeta},\ldots,\bar{\beta}_k/\bar{\zeta},
  -\bar{\alpha}_1/\bar{\zeta}, \ldots, -
  \bar{\alpha}_k/\bar{\zeta},-1/\bar{\zeta}) 
\end{equation*}
while \( \omega \) is given by \( \sum_{j=1}^k \mathrm{d}\alpha_j
\wedge \mathrm{d}\beta_j/2 \).

\section{The nilpotent cone}
\label{sec:nilpotent-cone}
In this section we recall the hyperk\"ahler structure on the nilpotent
cone in the Lie algebra of \( K_\C=\SL(n, \C) \) obtained in \cite{KS},
and describe its twistor space which can be embedded in the vector
bundle \( \cO(2) \otimes \lie k_\C \) over \( \PP^1 \). 

 The nilpotent
cone for \( K_\C = \SL(n,\C) \) is identified in \cite{KS} with a
hyperk\"ahler quotient \( M\hkq \tilde{H} \), where \( M \) is a flat
hyperk\"ahler space and \( \tilde{H} \) is a product of unitary groups
acting on \( M \). 

 Let us choose integers \( 0 \leqslant n_1
\leqslant n_2 \leqslant \dots \leqslant n_r=n \) and consider the flat
hyperk\"ahler space
\begin{equation}
  \label{eq:Mn}
  M = M(\mathbf n) = \bigoplus_{i=1}^{r-1} \HH^{n_i n_{i+1}} = 
  \bigoplus_{i=1}^{r-1} \Hom(\C^{n_i},\C^{n_{i+1}}) \oplus
  \Hom(\C^{n_{i+1}},\C^{n_i})
\end{equation}
with the hyperk\"ahler action of \( \Un(n_1) \times \dots \times
\Un(n_r) \)
\begin{equation*}
  \alpha_i \mapsto g_{i+1} \alpha_i g_i^{-1},\quad
  \beta_i \mapsto g_i \beta_i g_{i+1}^{-1} \qquad (i=1,\dots r-1),
\end{equation*}
with \( g_i \in \Un(n_i) \) for \( i=1, \dots, r \). Here \( \alpha_i
\) and \( \beta_i \) denote elements of \( \Hom (\C^{n_i},\allowbreak
\C^{n_{i+1}}) \) and \( \Hom (\C^{n_{i+1}},\C^{n_i}) \) respectively,
and right quaternion multiplication is given by
\begin{equation}
  \label{eq:j}
  (\alpha_i,\beta_i)\jj = (-\beta_i^*,\alpha_i^*).
\end{equation}
We may write \( (\alpha,\beta) \in M(\mathbf n) \) as a quiver
diagram:
\begin{equation} \label{eqn2.3} 0
  \stackrel[\beta_0]{\alpha_0}{\rightleftarrows}
  \C^{n_1}\stackrel[\beta_1]{\alpha_1}{\rightleftarrows}
  \C^{n_2}\stackrel[\beta_2]{\alpha_2}{\rightleftarrows}\dots
  \stackrel[\beta_{r-2}]{\alpha_{r-2}}{\rightleftarrows} \C^{n_{r-1}}
  \stackrel[\beta_{r-1}]{\alpha_{r-1}}{\rightleftarrows} \C^{n_r} =
  \C^n,
\end{equation}
where \( \alpha_0 = \beta_0 = 0 \).  For brevity, we will often call
such a diagram a quiver.  Let \( \tH \) be the subgroup, isomorphic to
\( \prod_{i=1}^{r-1} \Un(n_i) \), given by setting \( g_r=1 \) and let
\begin{gather*}
  \tmu \colon M \to \LIE(\tH) \otimes \R^3 =
  \LIE(\tH) \otimes (\R \oplus \C)\\
  \tmu(\alpha,\beta) =
  \bigl((\alpha_i\alpha_i^*-\beta_i^*\beta_i+\beta_{i+1}\beta_{i+1}^*
  - \alpha_{i+1}^*\alpha_{i+1})\ii, \alpha_i\beta_i -
  \beta_{i+1}\alpha_{i+1}\bigr)
\end{gather*}
be the corresponding hyperk\"ahler moment map. 

It is proved in \cite{KS} that when we have a full flag (that is, when
\( r=n \) and \( n_j=j \) for each \( j \), so that the centre of \(
\tH \) can be identified with the maximal torus \( T \) of \( K=\SU(n)
\)) then the hyperk\"ahler quotient \( \tmu^{-1}(0)/\tH \) of \( M \)
by \( \tH \) can be identified with the nilpotent cone in \( \kf_\C
\).

This hyperk\"ahler quotient carries an \( \SU(n) \) action induced
from the action of this group on the top space \( \C^n \) of the
quiver.  In \cite{KS} Theorem 2.1 it is shown that, for any choice of
complex structure, the complex moment map \( M\hkq \tilde{H} \to \lie
k _\C \) for this action induces a bijection from \( M\hkq \tilde{H}
\) onto the nilpotent cone in \( \lie k _\C \). This moment map is
given, with complex structure as above, by
\begin{equation*}
  (\alpha, \beta) \mapsto \alpha_{n-1} \beta_{n-1}.
\end{equation*}
Thus the hyperk\"ahler moment map \( \mu\colon M\hkq \tilde{H} \to
\lie k \otimes \R^3 \) provides a bijection from \( M\hkq \tilde{H} \)
to its image in \( \lie k \otimes \R^3 \), and this image is a \( K
\times \SU(2) \)-invariant subset \( \Nil(K) \) of \( \lie k \otimes
\R^3 \) such that after acting by \emph{any} element of \( \SU(2) \)
the projection \( \Nil(K) \to \lie k _\C \) given by the decomposition
\( \R^3 = \R \oplus \C \) is a bijection onto the nilpotent cone in \(
\lie k _\C \).  

The hyperk\"ahler structure on the nilpotent cone in
\( \lie k _\C \) can in principle be determined explicitly from the
bijection from \( M\hkq \tilde{H} \) given by \( (\alpha, \beta)
\mapsto \alpha_{n-1} \beta_{n-1} \), but it is not very easy to write
down a lift to \( \mu^{-1}(0) \) of the inverse of this bijection.
However, the hyperk\"ahler structure can be determined explicitly from
the embedding of \( \Nil(K) \) in \( \lie k \otimes \R^3 \) as
follows. The complex and complex-symplectic structures on \( \Nil(K)
\) are given by pulling back the standard complex and
complex-symplectic structures on nilpotent orbits in \( \lie k_\C \)
under the projections of \( \lie k \otimes \R^3 \) onto \( \lie k
\otimes \C \) corresponding to the different choices of complex
structures, and these determine the metric.

Thus it is useful to describe as explicitly as possible the embedding
of \( \Nil(K) \) in \( \lie k \otimes \R^3 \). If we fix a complex
structure and use it to identify \( \Nil(K) \) with the nilpotent cone
in \( \lie k _\C \) and to identify \( \lie k \otimes \R^3 \) with \(
\lie k \oplus \lie k_\C \), then the embedding is given by
\begin{equation*}
  \eta \mapsto ( \Phi_n(\eta), \eta)
\end{equation*}
where the map \( \Phi_n \) from the nilpotent cone in \( \sL(n,\C) \)
to \( \su(n) \) is the (real) moment map for the action of \( K=\SU(n)
\) on the nilpotent cone with respect to the K\"ahler form determined
by the standard complex structure and the hyperk\"ahler metric. The
map \( \Phi_n \) is determined inductively by the properties given in
Lemma~\ref{lem2.3} below.

\begin{remark}
  Note that \( M \) has an \( \SU(2) \)-action which commutes with the
  action of \( \tH \) and rotates the complex structures on \( M
  \). This action descends to an \( \SU(2) \)-action on \( M \hkq \tH
  \) and hence on \( \Nil(K) \). The induced \( \SU(2) \)-action on \(
  \Nil(K) \) extends to the action on \( \lie k \otimes \R^3 \) given
  by the usual rotation action on \( \R^3 \). When \( \Nil(K) \) is
  identified with the nilpotent cone in \( \lie k_\C \) by projection
  from \( \lie k \otimes \R^3 \) to \( \lie k_\C \), the action of
  \begin{equation*}
    \begin{pmatrix}
      u & v \\
      -\bar{v} & \bar{u}
    \end{pmatrix} \in \SU(2),
  \end{equation*}
  where \( \abs u^2 + \abs v^2 = 1 \), is given by
  \begin{equation*}
    \eta \mapsto u^2 \eta + 2uv \Phi_n(\eta) - v^2 \bar{\eta}^T.
  \end{equation*}
\end{remark}

\begin{lemma}
  \label{lem2.3}
  The map \( \Phi_n \) from the nilpotent cone in \( \lie k_\C =
  \sL(n,\C) \) to \( \lie k = \su(n) \) is \( \SU(n) \)-equivariant
  for the adjoint action of \( \SU(n) \) on \( \sL(n,\C) \) and \(
  \su(n) \). Furthermore if \( A \) is a generic upper triangular \(
  (n-1)\times(n-1) \) complex matrix with positive real eigenvalues so
  that the \( n\times n \) matrix
  \begin{equation*}
    \begin{pmatrix}
      0 & A^2\\
      0 & 0
    \end{pmatrix}
  \end{equation*}
  is strictly upper triangular, then
  \begin{equation*}
    \Phi_n\left(
      \begin{pmatrix}
        0 & A^2\\
        0 & 0 \end{pmatrix} \right) = \left( 
    \begin{pmatrix}
      AB(\overline{A} \overline{B})^T & 0\\
      0 & 0
    \end{pmatrix} - 
    \begin{pmatrix}
      0 & 0\\
      0 & (\overline{B}^{-1}\overline{A})^T B^{-1}A
    \end{pmatrix} \right) \ii
  \end{equation*}
  where \( B \) is upper triangular with real positive eigenvalues and
  satisfies
  \begin{equation*}
    \Phi_{n-1} \left(
      \begin{pmatrix}
        0 & B^{-1}A
      \end{pmatrix}
      \begin{pmatrix}
        AB\\
        0
      \end{pmatrix}\right)
    = ( B^{-1} A (\overline{B}^{-1} \overline{A})^T -  (\overline{A}
    \overline{B})^T AB) \ii.
  \end{equation*}
  These properties determine the continuous map \( \Phi_n \)
  inductively.
\end{lemma}

\begin{remark}
  In order for the upper triangular matrix \( A \) to be generic in the
  sense of this lemma, it is enough for \( A \) to lie in the regular
  nilpotent orbit in \( \lie k_\C \).
\end{remark}
\begin{proof}
  First note that any nilpotent matrix lies in the \( \SU(n)
  \)-adjoint orbit of a strictly upper triangular matrix with
  non-negative real entries immediately above the leading diagonal,
  and a generic such matrix (lying in the regular nilpotent orbit in
  \( \lie k_\C \)) can be expressed in the form
  \begin{equation*}
    \begin{pmatrix}
      0 & A^2\\
      0 & 0
    \end{pmatrix}
  \end{equation*}
  where \( A \) has positive real eigenvalues.  The hyperk\"ahler
  quotient \( M\hkq \tilde{H} \) can be described as \(
  \mu_{\tilde{H}}^{-1}(0)/\tilde{H} \) where \( \mu_{\tilde{H}} \) is
  the hyperk\"ahler moment map for the action of \( \tilde{H} \) on \(
  M \), and also as the GIT quotient of \(
  (\mu_{\tilde{H}})_\C^{-1}(0) \) by its complexification \(
  \tilde{H}_\C \).  A quiver \( (\alpha,\beta) \in M(\mathbf{n}) \) as
  at \eqref{eqn2.3} lies in \( \mu_{\tilde{H}}^{-1}(0)/\tilde{H} \) if
  and only if \( \beta_j \alpha_j = \alpha_{j-1} \beta_{j-1} \) for \(
  0 < j < n \).  Writing
  \begin{equation*}
    \begin{pmatrix}
      0 & A^2\\
      0 & 0
    \end{pmatrix}
    = \alpha_{n-1} \beta_{n-1}
  \end{equation*}
  where
  \begin{equation*}
    \alpha_{n-1} =
    \begin{pmatrix}
      A\\
      0
    \end{pmatrix}
    \quad\text{and}\quad \beta_{n-1} = 
    \begin{pmatrix}
      0 & A
    \end{pmatrix}
    ,
  \end{equation*}
  so that
  \begin{equation*}
    \beta_{n-1} \alpha_{n-1} = 
    \begin{pmatrix}
      0 & A
    \end{pmatrix}
    \begin{pmatrix}
      A\\
      0
    \end{pmatrix}
  \end{equation*}
  is a strictly upper triangular matrix with non-negative real entries
  immediately above the leading diagonal, allows us inductively to
  find a quiver \( (\alpha, \beta) \) in \(
  (\mu_{\tilde{H}})_\C^{-1}(0) \) such that
  \begin{equation} \label{eqnphi}
    \alpha_{n-1} \beta_{n-1} = 
    \begin{pmatrix}
      0 & A^2\\
      0 & 0
    \end{pmatrix}
  \end{equation}
  for generic \( A \) as above.  In order to find the value of \(
  \Phi_n \) on this matrix, however, we need to find a representative
  quiver in \( \mu_{\tilde{H}}^{-1}(0) \), since
  \begin{equation*}
    \Phi_n\left(
      \begin{pmatrix}
        0 & A^2\\
        0 & 0
      \end{pmatrix}
    \right)
    = ( \alpha_{n-1} \alpha^*_{n-1} -  \beta_{n-1}^* \beta_{n-1})\ii
  \end{equation*}
  for any quiver \( (\alpha,\beta) \in \mu_{\tilde{H}}^{-1}(0) \) such
  that \eqref{eqnphi} holds, and then
  \begin{equation*}
    \Phi_j(\alpha_{j-1} \beta_{j-1} ) = ( \alpha_{j-1} \alpha^*_{j-1} -  
    \beta_{j-1}^* \beta_{j-1} )\ii
  \end{equation*}
  for \( 0<j<n \).  The theory of moment maps and GIT quotients tells
  us that there will exist such a representative quiver in the closure
  of the \( \tilde{H}_\C \)-orbit of the quiver we have already found,
  and that for generic \( A \) this orbit is closed.  Indeed since \(
  \tilde{H}_\C \) can be expressed as the product of a Borel subgroup
  and its maximal compact subgroup \( \tilde{H} \), we can find such a
  representative quiver in the intersection of the Borel orbit and \(
  \mu_{\tilde{H}}^{-1}(0) \). The maps \( \alpha_{n-1} \) and \(
  \beta_{n-1} \) in this quiver will be given by matrices of the form
  \begin{equation*}
    \begin{pmatrix}
      AB\\
      0
    \end{pmatrix}
    \quad\text{and}\quad
    \begin{pmatrix}
      0 & B^{-1}A
    \end{pmatrix}
  \end{equation*}
  where we can assume that \( B \) is an upper triangular \( (n-1)
  \times (n-1) \) complex matrix with positive real eigenvalues.  As
  the quiver lies in \( \mu_{\tilde{H}}^{-1}(0) \) it follows that
  \begin{equation*}
    \Phi_{n-1} \left(
      \begin{pmatrix}
        0 & B^{-1}A
      \end{pmatrix}
      \begin{pmatrix}
        AB\\
        0
      \end{pmatrix}
    \right)
    = ( B^{-1} A (\overline{B}^{-1} \overline{A})^T -  (\overline{A}
    \overline{B})^T AB) \ii. 
  \end{equation*}
  This completes the proof.
\end{proof}
\begin{example}
  Consider the case when \( n=2 \). If \( d  \) is real and positive then
  \begin{equation*}
    \begin{pmatrix}
      0 & d\\
      0 & 0
    \end{pmatrix}
    =
    \begin{pmatrix}
      \sqrt{d}\\
      0
    \end{pmatrix}
    \begin{pmatrix}
      0 & \sqrt{d}
    \end{pmatrix}
    ,
  \end{equation*}
  where
  \begin{equation*}
    \begin{pmatrix}
      0 & \sqrt{d}
    \end{pmatrix}
    \begin{pmatrix}
      \sqrt{d}\\
      0
    \end{pmatrix}
    = 0,
  \end{equation*}
  and \( \Phi_1 = 0 \) so
  \begin{equation*}
    \Phi_2\left(
      \begin{pmatrix}
        0 & d\\
        0 & 0
      \end{pmatrix}
    \right) =  \begin{pmatrix}
      d & 0\\
      0 & -d
    \end{pmatrix} \ii
    .
  \end{equation*}
\end{example}

\begin{example}
  Let \( n=3 \) and let
  \begin{equation*}
    A= \begin{pmatrix}
      a & b\\
      0 & c
    \end{pmatrix}
    \quad\text{and}\quad  B= \begin{pmatrix}
      \alpha & \beta \\
      0 & 1/\gamma
    \end{pmatrix}
  \end{equation*}
  where \( a,c,\alpha,\gamma \) are real and positive and \(
  b,\beta\in \C \).  Then using the previous example we have
  \begin{equation*}
    \Phi_{2} \left( \begin{pmatrix}
        0 & B^{-1}A
      \end{pmatrix}
      \begin{pmatrix}
        AB\\
        0
      \end{pmatrix}
    \right) =
    \Phi_2\left( \begin{pmatrix}
        0 & ac/\alpha\gamma \\
        0 & 0
      \end{pmatrix}
    \right) =  
    \begin{pmatrix}
      ac/\alpha\gamma & 0\\
      0 & -ac/\alpha\gamma
    \end{pmatrix} \ii
  \end{equation*}
  while
  \begin{multline*}
    ( B^{-1} A (\overline{B}^{-1} \overline{A})^T - (\overline{A}
    \overline{B})^T AB) \ii
    \\
    = \begin{pmatrix} a^2(1/\alpha^2 - \alpha^2) + \abs{(b -
      c\beta\gamma)/\alpha}^2 &
      c\gamma(b - c\beta\gamma)/\alpha - a\alpha(a\beta + b/\gamma)\\[1ex]
      c\gamma\overline{(b - c\beta\gamma)}/\alpha -
      a\alpha\overline{(a\beta + b/\gamma)} & c^2(\gamma^2 -
      1/\gamma^2) - \abs{a\beta + b/\gamma}^2
    \end{pmatrix} \ii .
  \end{multline*}
  When these are equal we have
  \begin{equation*}
    \beta = \frac{b(c\gamma^2 - a \alpha^2)}{\gamma(a^2 \alpha^2 + c^2 \gamma^2)}
  \end{equation*}
  and
  \begin{equation*}
    a^2(1/\alpha^2 - \alpha^2) + \abs*{\frac{ab(a+c) \alpha}{a^2 \alpha^2 +
    c^2 \gamma^2}}^2 
    = ac/\alpha\gamma = c^2(1/\gamma^2 - \gamma^2) + \abs*{\frac{bc(a+c) \gamma}{a^2 \alpha^2 + c^2 \gamma^2}}^2.
  \end{equation*}
  It is convenient to write \( \alpha/\gamma = \delta \); then these
  equations are equivalent to
  \begin{equation} \label{bdelta} \abs b = \frac{a^2 \delta^2 +
    c^2}{a+c} \sqrt{\gamma^4 + \frac{a^2/\delta^2 - c^2}{c^2 - a^2
    \delta^2}}
  \end{equation}
  and \( f_{(a/c)}(\delta) = 0 \) where
  \begin{equation*}
    f_{(a/c)}(x) = x^4 - (a/c)x^3 + (c/a) x - 1.
  \end{equation*}
  Note that this polynomial factorises as \( (x - (a/c))(x^3 + (c/a)) \)
  so it has a unique positive root \( \delta = a/c \). Thus
  \begin{equation*}
    \gamma = c \left( \frac{(a+c)\abs b}{a^4 + c^4} 
    \right)^{1/2}
  \end{equation*}
  and \( \alpha = \frac{a \gamma}{c} \) and \( \beta =
  \frac{b(c^3-a^3)}{\gamma (a^4 + c^4)} = \frac{\gamma b (c^3-a^3)}
  {\abs{b} c^2 (a+c)}\).  Then
  \begin{equation*}
    \Phi_3\left( \begin{pmatrix}
        0 & A^2\\
        0 & 0 \end{pmatrix} \right)
  \end{equation*}
  is given by substituting these expressions for \(
  \alpha,\beta,\gamma \) into the matrix
  \begin{equation*}
    \begin{pmatrix}
      a^2 \alpha^2 + \abs*{a\beta + b/\gamma}^2 &
      (a\beta\gamma + b)c/\gamma^2 & 0 \\[1ex]
      (a\bar{\beta}\gamma + \bar{b})c/\gamma^2 & c^2/\gamma^2 -
      a^2/\alpha^2 & - a(b \alpha - \beta \gamma c)/\alpha^2 \\[1ex]
      0 & - a(\bar{b} \alpha - \bar{\beta} \gamma c)/\alpha^2 &
      \abs*{b - \beta \gamma c}^2/\alpha^2 -c^2 \gamma^2
    \end{pmatrix}
    \ii.
  \end{equation*}
  We obtain
  \begin{equation*}
    \begin{pmatrix}
      \abs{b} (a+c) & \frac{bc^2}{\abs{b}} & 0 \\
      \frac{\bar{b} c^2}{\abs{b}} & 0 & -\frac{b a^2}{\abs{b}} \\
      0 & -\frac{\bar{b} a^2}{\abs{b}} & -\abs{b} (a+c)
    \end{pmatrix} \ii
  \end{equation*}
  (cf. \cite[\S5]{KS2}).
\end{example}
Since \( M \) is a flat hyperk\"ahler manifold, it follows as in
\S \ref {sec:twistor-spaces} that its twistor space \( \twist_M \) is the
vector bundle
\begin{equation*}
  \cO(1) \otimes \bigoplus_{i=1}^{r-1} (\Hom(\C^{n_i},\C^{n_{i+1}}) \oplus
  \Hom(\C^{n_{i+1}},\C^{n_i}))
\end{equation*}
over \( \PP^1 \).  The complex moment map for the action of \(
\tilde{H} \) defines a morphism \( \twist_\mu \) from \( \twist_M \)
into the vector bundle \( \cO(2) \otimes \LIE\tilde{H}_\C \) over \(
\PP^1 \), and the twistor space for the hyperk\"ahler quotient \(
M\hkq \tilde{H} \) is the quotient (in the sense of geometric
invariant theory) by the action of the complexification \(
\tilde{H}_\C = \prod_{k=1}^{n-1} \GL(k) \) on the zero section of this
morphism.  The complex moment map for the action of \( K=SU(n) \)
defines a \( \tilde{H}_\C \)-invariant morphism from \( \twist_M \)
into the vector bundle \( \cO(2) \otimes \lie k_\C \) over \( \PP^1
\), and this induces an embedding into \( \cO(2) \otimes \lie k_\C \)
of the GIT quotient which is the twistor space \( \twist_{\Nil(K)} \)
for \( \Nil(K) = M\hkq \tilde{H} \).  It follows from \cite{KS}
Theorem 2.1 that this is an isomorphism of the twistor space with the
closed subvariety of the vector bundle \( \cO(2) \otimes \lie k_\C \)
over \( \PP^1 \) which meets each fibre in the tensor product of the
corresponding fibre of \( \cO(2) \) with the nilpotent cone in \( \lie
k_\C \).

The real structure on the twistor space \( \twist_{\Nil(K)} \) is
represented in local coordinates as at \eqref{labelcoord} by
\begin{equation*}
  (\alpha_{n-1}\beta_{n-1}, \zeta) \mapsto
  (-\overline{\beta}_{n-1}^T\overline{\alpha}_{n-1}^T/\bar{\zeta}^2, -
  1/\bar{\zeta}) 
\end{equation*}
or equivalently
\begin{equation*}
  (X,\zeta) \mapsto (-\bar{X}^T/\bar{\zeta}^2, -1/\bar{\zeta})
\end{equation*}
for \( X \) in the nilpotent cone in \( \lie k_\C \).

The remaining structure required for the twistor space \(
\twist_{\Nil(K)} \) is a holomorphic section \( \omega \) of \(
\wedge^2T_F^* \otimes \cO(2) \) where \( T_F \) denotes the tangent
bundle along the fibres of \( \pi\colon \twist_{\Nil(K)} \to \PP^1 \);
or rather, this is what would be required if \( \Nil(K) \) were
smooth. In fact \( \Nil(K) \) is singular with a stratified
hyperk\"ahler structure, where the strata are given by the (finitely
many) (co)adjoint orbits in the nilpotent cone in \( \lie k_\C \cong
\lie k_\C^* \), and \( \omega \) restricts on each stratum \( \Sigma
\) to a holomorphic section \( \omega_\Sigma \) of \(
\wedge^2T_{F,\Sigma}^* \otimes \cO(2) \) where \( T_{F,\Sigma} \) is
the tangent bundle along the fibres of the restriction of \( \pi \) to
the twistor space \( \twist_\Sigma \).

Recall that any coadjoint orbit \( \cO_\eta \cong K/K_\eta \) for \(
\eta \in \lie k^* \) of a compact group \( K \) has a canonical \( K
\)-invariant symplectic form, the Kirillov-Kostant form \( \omega_\eta
\), which can be obtained by symplectic reduction at \( \eta \) from
the cotangent bundle \( T^*K \) with its canonical symplectic
structure.  The \( K \)-invariant symplectic form \( \omega_\eta \) is
characterised by the property that the corresponding moment map for
the action of \( K \) on \( \cO_\eta \) is the inclusion of \(
\cO_\eta \) in \( \lie k^* \).

Similarly a coadjoint orbit \( \cO_\eta \) for \( \eta \in \lie k^*_\C
\) of the complexification \( K_\C \) of \( K \) has a canonical \(
K_\C \)-invariant holomorphic symplectic form \( \omega_\eta \) which
is again characterised by the property that the associated complex
moment map for the action of \( K_\C \) is the inclusion of \(
\cO_\eta \) in \( \lie k^*_\C \). 

If \( \Sigma \) is a stratum of \( {\Nil(K)} \) given by a coadjoint
orbit in \( \lie k^*_\C \) then the holomorphic section \(
\omega_\Sigma \) of \( \wedge^2T_{F,\Sigma}^* \otimes \cO(2) \) which
restricts to a holomorphic symplectic form on each fibre of \(
\pi\colon \twist_{\Sigma} \to \PP^1 \) is \( K_\C \)-invariant for \(
K=\SU(n) \), and the corresponding twistor moment map \(
\twist_{\Sigma} \to \lie k_\C^* \otimes \cO(2) \) is the restriction
of the embedding of \( \twist_{\Nil(K)} \) into \( \lie k_\C^* \otimes
\cO(2) \). Thus it follows that each \( \omega_\Sigma \) (and hence
also the holomorphic section \( \omega \) of \( \wedge^2T_F^* \otimes
\cO(2) \)) is given by the Kirillov-Kostant construction on the fibres
of \( \pi \).

The Springer resolution of the nilpotent cone in \( \lie k_\C \) is
given by the complex moment map for the action of \( \lie k_\C \) on
the cotangent bundle \( T^*\mathcal{B} \) where \( \mathcal{B} \) is
the flag manifold \( K_\C/B = K/T \) identified with the space of
Borel subgroups of \( K_\C \) \cite{Chriss-G:representation}. The
twistor space \( \twist_{\Nil(K)} \) has a corresponding resolution of
singularities
\begin{equation*}
  \tilde{\twist}_{\Nil(K)} \cong  K_\C \times_B (\lie b^0 \times \cO(2))
  \cong  K \times_T (\lie t^0 \times \cO(2))
\end{equation*}
where \( B \) is a Borel subgroup of \( K_\C \) containing \( T \) and
\( \lie b^0 \) is the annihilator of its Lie algebra in \( \lie k_\C^*
\) while \( \lie t^0 \) is the annihilator of the Lie algebra of \( T
\) in \( \lie k^* \).

\begin{remark}
  \label{nilkhnought}
  The twistor moment map \( \twist_{\Nil(K)} \to \lie k^*_\C \otimes
  \cO(2) \) restricts to a holomorphic section of \( \lie k^*_\C
  \otimes \cO(2) \) on each twistor line \( \{m\} \times \PP^1 \) in
  \( \twist_{\Nil(K)} \), and gives us a map
  \begin{equation*}
    \phi\colon {\Nil(K)} \to H^0(\PP^1, \lie k_\C^* \otimes \cO(2))
  \end{equation*}
  as in Remark~\ref{hnought}. For any \( p \in \PP^1 \) the composition of
  \( \phi \) with the evaluation map
  \begin{equation*}
    H^0(\PP^1, \lie k_\C^* \otimes \cO(2)) \to  \lie k_\C^* \otimes
    \cO(2)_p 
  \end{equation*}
  is injective, and its image is the nilpotent cone in \( \lie k^*_\C
  = \lie k_\C \) if we choose any basis for the one-dimensional
  complex vector space \( \cO(2)_p \) to identify \( \lie k_\C^*
  \otimes \cO(2)_p \) with \( \lie k_\C^* \). If we fix the complex
  structure corresponding to \( [1:0] \in \PP^1 \) to identify \(
  {\Nil(K)} \) with the nilpotent cone in \( \lie k_\C^* \), then \(
  \phi \) is given by
  \begin{equation*}
    \phi(\eta)[u:v] = u^2 \eta + 2uv \Phi_n(\eta) - v^2 \bar{\eta}^T
  \end{equation*}
  when \( \eta \in \lie k_\C \) is nilpotent and \( [u:v] \in \PP^1 \)
  and \( \Phi_n \) is as at Lemma~\ref{lem2.3}.

  The map \( \phi\colon {\Nil(K)} \to H^0(\PP^1, \lie k_\C^* \otimes
  \cO(2)) \) induces an embedding
  \begin{equation*}
    \phi_\twist \colon \twist_{\Nil(K)} = \PP^1 \times {\Nil(K)} \to
    \PP^1 \times H^0(\PP^1, \lie k_\C^* \otimes \cO(2))
  \end{equation*}
  which is {\em not} holomorphic, as well as the holomorphic embedding
  \begin{equation*}
    \twist_{\Nil(K)} \to \lie k_\C^* \otimes \cO(2).
  \end{equation*}
  This map (which may also be viewed as the twistor moment map for the
  \( K_\C \) action), is the composition of \( \phi_\twist \) with the
  natural map
  \begin{equation*}
    \PP^1 \times
    H^0(\PP^1, \lie k_\C^* \otimes \cO(2)) \to \lie k_\C^* \otimes \cO(2).
  \end{equation*}
  The real structure on \( \twist_{\Nil(K)} \) extends to the real
  structure on \( \PP^1 \times H^0(\PP^1, \lie k_\C^* \otimes \cO(2))
  \) determined by the standard real structure \( \zeta \mapsto
  -1/\bar{\zeta} \) on \( \PP^1 \) and the real structure \( \eta
  \mapsto -\bar{\eta}^T \) on \( \lie k_\C \).

  If \( X \) is any hyperk\"ahler manifold on which \( K = \SU(n) \)
  acts with hyperk\"ahler moment map \( \mu_X\colon X \to \lie k^*
  \otimes \R^3 \) and corresponding
  \begin{equation*}
    \phi_X\colon X \to H^0(\PP^1, \lie k_\C^* \otimes \cO(2))
  \end{equation*}
  as in Remark~\ref{hnought}, then we obtain a stratified hyperk\"ahler
  space
  \begin{equation*}
    X_{\mathrm{nil}} = (X \times {\Nil(K)}) \hkq K = \bigsqcup_{\cO}
    \phi_X^{-1}(\phi(\cO))/K
  \end{equation*}
  where the strata are indexed by the nilpotent coadjoint orbits \(
  \cO \) in \( \lie k_\C^* \). Since \( \phi_X \) and \( \phi \) are
  \( K \)-equivariant and \( K_\C = K P_{\cO}\), the stratum indexed
  by \( \cO \) is given by
  \begin{equation*}
    \phi_X^{-1}(\phi(\cO))/K \cong \phi_X^{-1}(\phi(P_{\cO}
    \eta_{\cO})) /K_{\cO},
  \end{equation*}
  where \( \eta_{\cO} \) is a representative of the orbit \( \cO \) in
  Jordan canonical form, while \( P_{\cO} \) is the associated
  Jacobson-Morosov parabolic (see \cite{Collingwood-McGovern} Remark
  3.8.5) and \( K_{\cO} = K \cap P_\cO \).

  It follows from \cite{DKS} Theorem 7.18 that \( \Nil(K) \) with any
  of its complex structures is the non-reductive GIT quotient (in the
  sense of \cite{DK}) of \( K_\C \times \lie b^0 \) by the Borel \( B
  = T_\C N \) in \( K_\C \). Hence if \( X \) is a complex affine
  variety with respect to any of its complex structures, then \(
  X_{\mathrm{nil}} \) is the complex symplectic quotient (with respect
  to non-reductive GIT) of \( X \) by the action of the Borel \( B \).
\end{remark}

\section{Symplectic and hyperk\"ahler implosion}
\label{sec:symplectic-implosion}

In this section we recall the constructions of symplectic implosion
from \cite{GJS} and hyperk\"ahler implosion for \( K=\SU(n) \) from
\cite{DKS,DKS2}.

Let \( M \) be a symplectic manifold with Hamiltonian action of a
compact group \( K \).  The imploded space \( M_\impl \) is a
stratified symplectic space with a Hamiltonian action of the maximal
torus \( T \) of \( K \). It has the property, that, denoting
symplectic reduction at \(\lambda\) by \( \symp^s_\lambda \),
\begin{equation}
  M \symp_\lambda^s K   = M_\impl \symp^s_\lambda T
\end{equation}
for all \( \lambda \) in the closure \( \tf_{+}^* \) of a fixed
positive Weyl chamber in \( \tf^* \).  In particular the implosion \(
(T^*K)_\impl \) of the cotangent bundle \( T^*K \) inherits a
Hamiltonian \( K \times T \)-action from the Hamiltonian \( K \times K
\)-action on \( T^*K \). This example is universal in the sense that
for a general \( M \) we have
\begin{equation*}
  M_\impl = (M \times (T^*K)_\impl) \symp^s_0 K.
\end{equation*}
Concretely, the implosion \( (T^*K)_\impl \) of \( T^*K \) with
respect to the right action is constructed from \( K \times \tf_{+}^*
\) by identifying \( (k_1, \xi) \) with \( (k_2, \xi) \) if \( k_1
k_2^{-1} \) lies in the commutator subgroup of the \( K \)-stabiliser
of \( \xi \). The identifications which occur are therefore controlled
by the face structure of the Weyl chamber. In particular if \( \xi \)
is in the interior of the chamber, its stabiliser is a torus and no
identifications are performed.  An open dense subset of \(
(T^*K)_\impl \), therefore, is the product of \( K \) with the
interior of the Weyl chamber.

As explained in \cite{GJS}, when \( K \) is a connected, simply
connected, semisimple compact Lie group, we may embed the universal
symplectic implosion \( (T^*K)_{\impl} \) in the complex affine space
\( E = \oplus V_{\varpi} \), where \( V_{\varpi} \) is the \( K
\)-module with highest weight \( \varpi \) and we sum over a minimal
generating set \( \Pi \) for the monoid of dominant weights.  Under
this embedding, the implosion is identified with the closure \(
\overline{K_\C v} \), where \( v \) is the sum of the highest weight
vectors \( v_{\varpi} \) of the \( K \)-modules \( V_{\varpi} \), and
as usual \( K_{\C} \) denotes the complexification of \( K \). This
gives an alternative, more algebraic, description of the implosion as
a stratified space.  For the stabiliser of \( v \) is a maximal
unipotent subgroup \( N \) of \( K_{\C} \) (that is, the commutator
subgroup \( [B,B] \) of the corresponding Borel subgroup \( B \)) and
hence we may regard \( K_\C v \) as \( K_{\C}/N \). The
lower-dimensional strata which we obtain by taking the closure are
just the quotients \( K_{\C}/[P,P] \) for standard parabolic subgroups
\( P \) of \( K_{\C} \). These standard parabolics are of course in
bijective correspondence with the faces of the Weyl chamber, and this
algebraic stratification is compatible with the symplectic
stratification described above.

The whole implosion may be identified with the Geometric Invariant
Theory (GIT) quotient of \( K_{\C} \) by the non-reductive group \( N
\):
\begin{equation*}
  K_\C \symp N = \Spec(\cO(K_\C)^N);
\end{equation*}
it is often useful to view this as the canonical affine completion of
the quasi-affine variety \( K_\C / N \) (cf. \cite{DK}).

Recalling the Iwasawa decomposition \( K_{\C} = KAN \), where \(
T_{\C}=TA \), we see that
\begin{equation}
  \label{toricsymplectic}
  \overline{K_\C v} = \overline{KAv} = K (\overline{T_\C v}) ,
\end{equation}
so the implosion is described as the sweep under the compact group \(
K \) of a toric variety \( \overline{T_\C v} \). This toric variety is
associated to the positive Weyl chamber \( \lie t_+ \), and is in fact
the subspace
\begin{equation*}
  E^N = \bigoplus_{\varpi \in \Pi} \C v_\varpi
\end{equation*}
of \( E \) spanned by the highest weight vectors \( v_\varpi \). The
canonical affine completion \( K_\C \symp N \) of \( K_C/N \) has a
resolution of singularities
\begin{equation*}
  \widetilde{K_\C \symp N} = K_\C \times_B E^N \to K_\C \symp N
\end{equation*}
induced by the group action \( K_\C \times E^N \to E \).

As explained in \cite{DKS} one can also construct the symplectic
implosion for \( K=SU(n) \) using \emph{symplectic quivers}. These are
diagrams
\begin{equation}
  \label{eq:symplectic}
  0 = V_0 \stackrel{\alpha_0}{\rightarrow}
  V_1 \stackrel{\alpha_1}{\rightarrow}
  V_2 \stackrel{\alpha_2}{\rightarrow} \dots
  \stackrel{\alpha_{r-2}}{\rightarrow} V_{r-1}
  \stackrel{\alpha_{r-1}}{\rightarrow} V_r = \C^n.
\end{equation}
where \( V_i \) is a vector space of dimension \( n_i \). We realised
the symplectic implosion as the GIT quotient of the space of full flag
quivers (i.e., where \( r=n \) and \( n_i=i \)), by \(
\prod_{i=1}^{r-1} \SL(V_i) \), where this group acts by
\begin{align*}
  \alpha_i &\mapsto g_{i+1} \alpha_i g_i^{-1} \quad (i = 1,\dots, r-2),\\
  \alpha_{r-1} &\mapsto \alpha_{r-1} g_{r-1}^{-1}.
\end{align*}

In \cite{DKS} a hyperk\"ahler analogue of the symplectic implosion was
introduced for the group \( K= SU(n) \). We consider quiver diagrams
of the following form:
\begin{equation*}
  0 = V_0\stackrel[\beta_0]{\alpha_0}{\rightleftarrows}
  V_1\stackrel[\beta_1]{\alpha_1}{\rightleftarrows}
  V_2\stackrel[\beta_2]{\alpha_2}{\rightleftarrows}\dots
  \stackrel[\beta_{r-2}]{\alpha_{r-2}}{\rightleftarrows}V_{r-1}
  \stackrel[\beta_{r-1}]{\alpha_{r-1}}{\rightleftarrows} V_r
  = \C^n 
\end{equation*}
where \( V_i \) is a complex vector space of complex dimension \( n_i
\) and \( \alpha_0 = \beta_0 =0 \). The space \( M \) of quivers for
fixed dimension vector \( (n_1, \dots, n_r) \) is a flat hyperk\"ahler
vector space.  

As discussed earlier, there is a hyperk\"ahler action
of \( \Un(n_1) \times \dots \times \Un(n_r) \) on this space given by
\begin{equation*}
  \alpha_i \mapsto g_{i+1} \alpha_i g_i^{-1},\quad
  \beta_i \mapsto g_i \beta_i g_{i+1}^{-1} \qquad (i=1,\dots r-1),
\end{equation*}
with \( g_i \in \Un(n_i) \) for \( i=1, \dots, r \).  Recall that we
defined \( \tilde{H} \) be the subgroup, isomorphic to \( \Un(n_1)
\times \dots \times \Un(n_{r-1}) \), given by setting \( g_r=1 \). We
also let \(H = SU(n_1) \times \dots \times SU(n_{r-1}) \leqslant
\tilde{H}\).
\begin{definition}
  \label{defQ}
  The \emph{universal hyperk\"ahler implosion for \( \SU(n) \)} is the
  hyperk\"ahler quotient \( Q = M \hkq H \), where \( M, H \) are as
  above with \( r=n \) and \( n_j =j \), for \( j=1, \dots, n \).
\end{definition}

This hyperk\"ahler quotient \( Q \) has a residual action of \(
(S^1)^{n-1} =\tilde{H}/H \) as well as an action of \( \SU(n_r) =
\SU(n) \). As explained in \cite{DKS} we may identify \( (S^1)^{n-1}
\) with the maximal torus \( T \) of \( SU(n) \).  There is also an \(
Sp(1)=\SU(2) \) action which is not hyperk\"ahler but rotates the
complex structures.

Using the standard theory relating symplectic and GIT quotients, we
have a description of \( Q=M \hkq H \), as the quotient (in the GIT
sense) of the subvariety defined by the complex moment map equations
\begin{align}
  \label{eq:mmcomplex}
  \alpha_i \beta_i - \beta_{i+1} \alpha_{i+1} &= \lambda^\C_{i+1} I
  \qquad (0 \leqslant i \leqslant r-2)
\end{align}
where \( \lambda^\C_i \) for \( 1 \leqslant i \leqslant r-1 \) are
complex scalars, by the action of
\begin{gather}
  H_\C = \prod_{i=1}^{r-1}\SL(n_i, \C) \notag \\
  \label{SLaction1}
  \alpha_i \mapsto g_{i+1} \alpha_i g_i^{-1}, \quad \beta_i \mapsto
  g_i \beta_i g_{i+1}^{-1} \qquad (i=1,\dots r-2),\\
  \label{SLaction2}
  \alpha_{r-1} \mapsto \alpha_{r-1} g_{r-1}^{-1}, \quad \beta_{r-1}
  \mapsto g_{r-1} \beta_{r-1},
\end{gather}
where \( g_i \in \SL(n_i, \C) \).

The element \( X = \alpha_{r-1} \beta_{r-1} \in \Hom (\C^n, \C^n) \)
is invariant under the action of \( \prod_{i=1}^{r-1}\GL(n_i,\C) \)
and transforms by conjugation under the residual \(
\SL(n,\C)=\SL(n_r,\C) \) action on \( Q \).  We thus have a \( T_{\C}
\)-invariant and \( \SL(n,\C) \)-equivariant map \( Q \rightarrow
\sL(n,\C) \) given by:
\begin{equation*}
  (\alpha, \beta) \mapsto (X)_0 = X - \frac1n \tr(X) I_n
\end{equation*}
where \( I_n \) is the \( n\times n \)-identity matrix. This is the
complex moment map for the residual \( SU(n) \) action.

It is shown in \cite{DKS} that \( X \) satisfies an equation
\begin{equation*}
  X(X+ \nu_1) \dots (X+ \nu_{n-1})=0
\end{equation*}
where \( \nu_i = \sum_{j=i}^{r-1} \lambda_j^\C \). This generalises
the equation \( X^n =0 \) in the quiver construction of the nilpotent
variety in \cite{KS}, and is a consequence of Lemma 5.9 from
\cite{DKS} which gives information about the eigenvalues of \( X \)
and other endomorphisms derived from final segments of the quiver.
Define
\begin{equation}
  \label{eq:defXk}
  X_k = \alpha_{r-1} \alpha_{r-2} \dots
  \alpha_{r-k} \beta_{r-k} \dots \beta_{r-2} \beta_{r-1} \qquad (1
  \leqslant k \leqslant r-1)
\end{equation}
so that \( X = X_1 \).  It is proved in \cite{DKS} Lemma 5.9 that
for \( (\alpha,\beta) \in \mu_\C^{-1}(0) \),
satisfying~\eqref{eq:mmcomplex}, we have
\begin{equation}
  \label{eq:Xformula}
  X_k X = X_{k+1} - (\lambda_{r-1}^\C + \dots +
  \lambda_{r-k}^\C) X_k. 
\end{equation}
It follows from this by induction on \( j \) that if \( 0 \leq j < k <
r \) then
\begin{equation*}
  X_k = X_{k-j}X^j + \sum_{i=1}^j \sigma_i(\nu_{r-k+1},\ldots,\nu_{r-k+j})X_{k-j}X^{j-i}
\end{equation*}
where \( \sigma_i \) denotes the \( i \)th elementary symmetric
polynomial.  In particular putting \( j=k-1 \) gives us
\begin{equation*}
  X_k = X^k + \sum_{i=1}^{k-1} \sigma_i(\nu_{r-k+1},\ldots,\nu_{r-1})X^{k-i}
  = X \prod_{i=1}^{k-1}(X+ \nu_{r-i})
\end{equation*}
so
\begin{equation*}
  X_k = \prod_{i=0}^{k-1}(X+ \sum_{j=1}^i \lambda^{\C}_{r-j}).
\end{equation*}
Thus we have
\begin{lemma}
  \label{DKSlem5.9} If \( 1 \leqslant k \leqslant r-1 \) then
  \begin{equation*}
    \wedge^{r-k}(\alpha_{r-1}\alpha_{r-2} \cdots \alpha_{r-k})
    \wedge^{r-k}(\beta_{r-k} \cdots \beta_{r-1}) =
    \wedge^{r-k}  \prod_{i=0}^{k-1}(X+ \sum_{j=1}^i \lambda^{\C}_{r-j}).
  \end{equation*}
\end{lemma}
Recall from \eqref{toricsymplectic} that the universal symplectic
implosion is the \( K_\C \)-sweep of a toric variety \( \overline{T_\C
v} \). In \cite{DKS2} we found a hypertoric variety mapping
generically injectively to the hyperk\"ahler implosion. Hypertoric 
varieties are hyperk\"ahler quotients of flat quaternionic spaces
\( \mathbb{H}^d \) by subtori of \( (S^1)^d \); for more background see 
\cite{BD,HS}.

\begin{definition}
  \label{defn3.10}
  Let \( M_T \) be the subset of \( M \) consisting of hyperk\"ahler
  quivers of the form:
  \begin{equation*}
    \alpha_k = \begin{pmatrix}
      0 & \cdots &  0 & 0 \\
      \nu_1^k & 0  & \cdots & 0\\
      0 & \nu_2^k  & \cdots & 0\\
      & & \cdots  & \\
      0 & \cdots  & 0 & \nu_k^k
    \end{pmatrix}
  \end{equation*}
  and
  \begin{equation*}
    \beta_k = \begin{pmatrix}
      0 & \mu_1^k & 0  &  \cdots & 0\\
      0 & 0 & \mu_2^k  &  \cdots & 0\\
      & & \cdots &  &  \\
      0 & 0 & \cdots & 0  & \mu_k^k   \end{pmatrix}
  \end{equation*}
  for some \( \nu^k_i, \mu_i^k \in \C \). Note that this definition of
  \( M_T \) differs slightly from the definition in \cite{DKS2} but
  only by the action of an element of the Weyl group of \( H \times K
  \).
\end{definition}
As explained in \cite{DKS2}, for quivers of this form the moment map
equations for the action of \( H \) reduce to the moment map equations
for the action of its maximal torus \( T_{H} \) which is the product
over all \( k \) from 1 to \( n-1 \) of the standard maximal tori in
\( SU(k) \). We say that the quiver is {\em hyperk\"ahler stable} if
it has all \( \alpha_i \) injective and all \( \beta_i \) surjective
after a suitable rotation of complex structures. For quivers of the
form above, this means that \( \mu_i^k \) and \( \nu_i^k \) do not
both vanish for any \( (i,k) \) with \( 1 \leq i \leq k < n \). Two
hyperk\"ahler stable quivers of this form satisfying the hyperk\"ahler
moment map equations lie in the same orbit for the action of \( H \)
if and only if they lie in the same orbit for the action of its
maximal torus \( T_{H} \). We therefore get a natural map \( \iota \)
from the hypertoric variety \( M_T \hkq T_H \) to the implosion \( Q =
M\hkq H \), which restricts to an embedding
\begin{equation*}
  \iota\colon Q^{\hks}_T \to Q
\end{equation*}
where \( Q^{\hks}_T = M_T^{\hks} \hkq T_H \) and \( M_T^{\hks} \)
denotes the hyperk\"ahler stable elements of \( M_T \). Let \( Q_T =
\iota(M_T \hkq T_H) \) be the image of \( \iota\colon M_T \hkq T_H \to
Q \).

The space \( M_T \) is hypertoric for the maximal torus \( T_{\tH} \)
of \( \tH \), and \( M_T \hkq T_H \) is hypertoric for the torus \(
T_{\tH} /T_H = \tH/H = (S^1)^{n-1} \) which can be identified with \(
T \) as in \cite[\S3]{DKS2}, in such a way that the induced action of
\( K \times T \) on \( Q \) restricts to an action of \( T \times T \)
on \( Q_T \) such that \( (t,1) \) and \( (1,t) \) act in the same way
on \( Q_T \) for any \( t \in T \).

Indeed, by \cite{DKS2} Remark 3.2, \( M_T\hkq T_H \) is the hypertoric
variety for \( T \) associated to the hyperplane arrangement in its
Lie algebra \( \lie t \) given by the root planes.
\begin{remark}
  \label{hypertoricT}
  The root planes in the Lie algebra of the maximal torus \(
  T_{\Un(n)} \) of \( \Un(n) \) are the coordinate hyperplanes in \(
  \mathrm{Lie}(T_{\Un(n)}) = \R^n \), and the corresponding hypertoric
  variety for \( T_{\Un(n)} \) is \( \HH^n \). Thus we can identify \(
  M_T \hkq T_H \) with the hypertoric variety
  \begin{equation*}
    \{(w_1, \ldots,w_n) \in \HH^n : w_1 + \cdots + w_n = 0 \}
  \end{equation*}
  for \( T = (S^1)^{n-1} \).
\end{remark}

\section{Towards an embedding of the universal hyperk\"ahler implosion
in a linear representation of \( K\times T \)}
\label{sec:towards-an-embedding}

Let $\ell_j = j(n-j)$. In this section we define a \( K \times T \times \SU(2) \)-equivariant
map \( \sigma \) from the universal hyperk\"ahler implosion \( Q \)
for \( K=\SU(n) \) to
\begin{equation*} \mathcal{R} = 
  H^0(\PP^1,
  (\cO(2) \otimes (\lie k_\C \oplus \lie t_\C)) \oplus 
  \bigoplus_{j=1}^{n-1}  \cO(\ell_j) \otimes \wedge^j\C^n)
\end{equation*}
and an associated holomorphic map \( \tilde{\sigma} \) from the
twistor space of \( Q \) to the vector bundle \( (\cO(2) \otimes (\lie
k_\C \oplus \lie t_\C)) \oplus \bigoplus_{j=1}^{n-1} \cO(\ell_j) \otimes
\wedge^j\C^n \) over \( \PP^1 \); the first of these maps is proved to
be injective and the second is proved to be generically injective
in~\S \ref {sec:embeddings}.

As in \S \ref {sec:symplectic-implosion}, the universal symplectic
implosion \( (T^*K)_{\impl} \) for \( K=\SU(n) \) has a canonical
embedding in a linear representation of \( K\times T \) associated to
its description as the non-reductive GIT quotient
\begin{equation*}
  K_\C \symp N = \Spec(\cO(K_\C)^N)
\end{equation*}
where \( N \) is a maximal unipotent subgroup of the complexification
\( K_\C = \SL(n,\C) \) of \( K \) (cf. \cite{GJS}). The highest
weights of the irreducible representations \( \C^n, \wedge^2 \C^n,
\ldots, \wedge^{n-1}\C^n \) of \( K \) generate the monoid of dominant
weights, and each \( \wedge^j\C^n \) becomes a representation of \( K
\times T \) when \( T \) acts as multiplication by the inverse of the
corresponding highest weight.  Then \( K_\C \symp N \) is embedded in
the representation
\begin{equation*}
  \C^n \oplus \wedge^2 \C^n \oplus \cdots \oplus \wedge^{n-1}\C^n
\end{equation*}
of \( K \times T \) as the closure of the \( K_\C \)-orbit of
\begin{equation*}
  \sum_{j=1}^{n-1} v_j
\end{equation*}
where \( v_j \in \wedge^j\C^n \) is a highest weight vector, fixed by
\( N \).

Similarly we expect the universal hyperk\"ahler implosion for \(
K=\SU(n) \) to have an embedding in a representation of \( K\times T
\). In this section we will define a map which will later be shown to
provide such an embedding. 

Let
\begin{equation}
  \label{quivernormal}
  0 \stackrel[\beta_0]{\alpha_0}{\rightleftarrows}
  \C\stackrel[\beta_1]{\alpha_1}{\rightleftarrows}
  \C^{2}\stackrel[\beta_2]{\alpha_2}{\rightleftarrows}\dots
  \stackrel[\beta_{n-2}]{\alpha_{n-2}}{\rightleftarrows}
  \C^{n-1}
  \stackrel[\beta_{n-1}]{\alpha_{n-1}}{\rightleftarrows}
  \C^n
\end{equation}
be a quiver in \( M \) which satisfies the hyperk\"ahler moment map
equations for
\begin{equation*}
  H = \prod_{k=1}^{n-1} \SU(k).
\end{equation*}
Recalling that \( \alpha_0=\beta_0=0 \), these equations are given by
\begin{equation}
  \alpha_i \beta_i - \beta_{i+1}
  \alpha_{i+1} = \lambda^\C_{i+1} I \qquad (0 \leqslant i \leqslant
  n-2)
\end{equation}
where \( \lambda^\C_i \in \C\) for \( 1 \leqslant i \leqslant n-1 \),
and
\begin{equation}
  \alpha_i \alpha_i^* - \beta_i^* \beta_i +
  \beta_{i+1} \beta_{i+1}^* - \alpha_{i+1}^* \alpha_{i+1} =
  \lambda^\R_{i+1} I \quad (0 \leqslant i \leqslant n-2), 
\end{equation}
where \( \lambda^\R_i \in \R\) for \( 1 \leqslant i \leqslant n-1 \).
Then
\begin{equation*}
  (\alpha_{n-1} \beta_{n-1})_0 = \alpha_{n-1} \beta_{n-1} - \frac1n
  \tr(\alpha_{n-1} \beta_{n-1})I_n \in \lie k_\C
\end{equation*}
is invariant under the action of \( H \), as is
\begin{equation*}
  \bigl((u\alpha_{n-1} + v \beta_{n-1}^*)(-v\alpha_{n-1}^* + u
  \beta_{n-1})\bigr)_0
\end{equation*}
for any \( (u,v) \in \C^2 \) representing an element of
\begin{equation*}
  \SU(2) = \left\{
    \begin{pmatrix}
      u & v \\ -\bar{v} & \bar{u}
    \end{pmatrix}
    : \abs u^2 + \abs v^2 = 1 \right\},
\end{equation*}
where \( \alpha_{n-1}^* \) and \( \beta_{n-1}^* \) denote the adjoints
of \( \alpha_{n-1} \) and \( \beta_{n-1} \).  The same is true of
\begin{equation*}
  \sigma_j^{(\alpha,\beta)} = \wedge^j( \alpha_{n-1} \alpha_{n-2} \cdots \alpha_j)
  \in \wedge^j \C^n
\end{equation*}
where the linear map \( \wedge^j( \alpha_{n-1} \alpha_{n-2} \cdots
\alpha_j)\colon \wedge^j\C^j \to \wedge^j \C^n \) is identified with
the image of the standard basis element of the one-dimensional complex
vector space \( \wedge^j\C^j \), and the element
\begin{equation*}
  \sigma_j^{(u\alpha + v\beta^*,-{v}\alpha^* + {u}\beta)} = \wedge^j(
  u\alpha_{n-1}+ v \beta_{n-1}^*)(u \alpha_{n-2}+ v \beta_{n-2}^*)
  \cdots (u\alpha_j + v\beta_j^*)
\end{equation*}
of \( \wedge^j \C^n \) is also \( H \)-invariant for any \( (u,v) \in
\C^2 \). Note that
\begin{equation*}
  \PP^1 = \SU(2)/S^1
\end{equation*}
where
\begin{equation*}
  S^1 = \left\{ \begin{pmatrix}
      t & 0 \\ 0 & \bar{t} \end{pmatrix}: \abs t = 1 \right\}
\end{equation*}
acts on \( \SU(2) \) by left multiplication, and that
\begin{equation} \label{eqn((1))} (u,v) \mapsto \bigl((u\alpha_{n-1} +
  v \beta_{n-1}^*)(-v\alpha_{n-1}^* + u \beta_{n-1})\bigr)_0
\end{equation}
defines an element of
\begin{equation*}
  H^0(\PP^1,\cO(2) \otimes \lie k_\C)
\end{equation*}
whose value on a point \( p \) of \( \PP^1 \) is the value at the
quiver \eqref{quivernormal} of the corresponding complex moment map
for the \( K \)-action on \( M \) (up to multiplication by a non-zero
complex scalar depending on a choice of basis of the fibre \( \cO(2)_p
\) of the line bundle \( \cO(2) \) at \( p \)). Similarly the map
\begin{equation} \label{eqn((2))} (u,v) \,\,\, \mapsto \,\,\, {u}^2
  \lambda^\C +  {u} v \lambda^\R - v^2
  \bar{\lambda}^\C \end{equation} defines an element of
\begin{equation*}
  H^0(\PP^1,\cO(2)) \otimes \lie t_\C
\end{equation*}
whose value on a point \( p \) of \( \PP^1 \) is the value at the
quiver \eqref{quivernormal} of the corresponding complex moment map
for the \( T \)-action on \( M \) (up to multiplication by a non-zero
complex scalar depending on a choice of basis of the fibre \( \cO(2)_p
\)). Again this map is defined in an \( H \)-invariant way.  Finally
if \( 1 \leqslant j \leqslant n-1 \) then
\begin{equation} \label{eqn((3))}
  \begin{multlined}
    (u,v) \mapsto \sigma_j^{(u\alpha + v\beta^*,-{v}\alpha^* +
    {u}\beta)} \\= \wedge^j( u\alpha_{n-1}+ v \beta^*_{n-1})(u
    \alpha_{n-2}+ v \beta_{n-2}^*) \cdots (u\alpha_j + v\beta_j^*) \in
    \wedge^j \C^n
  \end{multlined}
\end{equation}
defines an element of
\begin{equation*}
  H^0(\PP^1,\cO(\ell_j) \otimes \wedge^j\C^n).
\end{equation*}

\begin{definition}
  \label{defnsigma}
  Let $\ell_j = j(n-j)$ as before and let 
  \begin{equation*}
    \sigma \colon Q \to
    \mathcal{R} =  H^0(\PP^1,
    \cO(2) \otimes (\lie k_\C \oplus  \lie t_\C) \oplus 
    \bigoplus_{j=1}^{n-1}  \cO(\ell_j) \otimes \wedge^j\C^n)
  \end{equation*}
  be the map defined by combining \eqref{eqn((1))}, \eqref{eqn((2))}
  and \eqref{eqn((3))} above.
\end{definition}

\begin{remark}
  The projection of \( \sigma \) to \( H^0(\PP^1, \cO(2) \otimes (\lie
  k_\C \oplus \lie t_\C) ) \) is the map \( \phi \)
  associated as in Remark~\ref{hnought} to the action of \( K \times T \) on
  \( Q \).
\end{remark}
\begin{remark}
  \label{remsigmatilde}
  Note that \( \sigma \) is \( K \times T \times \SU(2) \)-equivariant
  when \( K \times T \times \SU(2) \) acts on \( \mathcal{R} \) as
  described in the introduction, and \( \SU(2) \) acts on \( Q \)
  (commuting with the actions of \( K \) and \( T \)) via the inclusion \(
  \SU(2) = \mathrm{Sp}(1) \leqslant \HH^* \).  Moreover the evaluation
  \begin{equation*}
    \sigma_p\colon Q \to \cO(2)_p \otimes (\lie k_\C \oplus
    \lie t_\C) \oplus 
    \bigoplus_{j=1}^{n-1}  \cO(\ell_j)_p \otimes \wedge^j\C^n
    \cong  \lie k_\C \oplus  \lie t_\C \oplus 
    \bigoplus_{j=1}^{n-1}  \wedge^j\C^n
  \end{equation*}
  of \( \sigma \) at any point \( p \) of \( \PP^1 \) is a morphism of
  complex affine varieties with respect to the complex structure on \(
  Q \) determined by that point of \( \PP^1 \). Furthermore the
  projection of \( \sigma_p \) to \( \lie k_\C \) and to \( \lie t_\C
  \) can be identified with the complex moment map for the complex
  structure associated to \( p \) for the action of \( K \) and of \(
  T \) on \( Q \), once we have fixed a basis element for the fibre \(
  \cO(2)_p \) of the line bundle \( \cO(2) \) at \( p \). Of course a
  choice of basis elements for the fibres \( \cO(\ell_j)_p \) of the line
  bundles \( \cO(\ell_j) \) at \( p \) for all \( j \geqslant 1 \) is
  determined canonically by a choice of basis element for the fibre \(
  \cO(1)_p \) of the line bundle \( \cO(1) \) at \( p \), and so \(
  \sigma_p \) determines a map \( Q \to \lie k_\C \oplus \lie t_\C
  \oplus \bigoplus_{j=1}^{n-1} \wedge^j\C^n \) canonically up to the
  action of \( \C^* \) with weight 2 on \( \lie k_\C \oplus \lie t_\C
  \) and weight \( \ell_j \) on \( \wedge^j\C^n \).
\end{remark}
\begin{remark}
  \label{kcbyn}
  Note that if \( u\alpha_j + v\beta_j^* \) is injective for each \( j
  \) then the projection of \( \sigma_{[u:v]} \) onto \(
  E=\bigoplus_{j=1}^{n-1} \wedge^j\C^n \) maps the quiver into the \(
  K_\C \)-orbit of the sum
  \begin{equation*}
    \sum_{j=1}^{n-1} v_j \in E^N \subseteq E
  \end{equation*}
  of highest weight vectors \( v_j \) in the fundamental
  representations \( \wedge^j\C^n \) of \( K=\SU(n) \) for \( j=1,
  \ldots, n-1 \). Since this condition is satisfied by generic quivers
  in \( Q \), it follows that the projection of \( \sigma_{[u:v]} \)
  onto \( E=\bigoplus_{j=1}^{n-1} \wedge^j\C^n \) maps \( Q \) into
  the canonical affine completion
  \begin{equation*}
    K_\C
    \symp N = \overline{K_\C \left(\sum_{j=1}^{n-1} v_j \right)}.
  \end{equation*}
  Indeed, recall from \cite{GJS} that \( K_\C \symp N \) is the union
  of finitely many \( K_\C \)-orbits, one for each face \( \tau \) of
  the positive Weyl chamber \( \lie t_+ \) for \( K \), with
  stabiliser the commutator subgroup \( [P_\tau,P_\tau] \) of the
  corresponding parabolic subgroup \( P_\tau \) of \( K_\C \) whose
  intersection with \( K \) is the stabiliser \( K_\tau \) of \( \tau
  \) under the (co-)adjoint action of \( K \).  It follows from
  Theorem 6.13 of \cite{DKS} that if \( q \in Q \) then for generic \(
  p \in \PP^1 \) the projection of \( \sigma_p \) onto \(
  E=\bigoplus_{j=1}^{n-1} \wedge^j\C^n \) lies in the \( K_\C \)-orbit
  in \( K_\C\symp N \) with stabiliser \( [P_\tau,P_\tau] \) where \(
  \tau \) is the face of \( \lie t_+ \) whose stabiliser \( K_\tau \)
  in \( K \) is the stabiliser \( K_\lambda \) of the image \( \lambda
  \in \lie k \otimes \R^3 \) of \( q \) under the hyperk\"ahler moment
  map for the action of \( T \) on \( Q \).
\end{remark}

We will prove in \S \ref {sec:embeddings}

\begin{theorem}
  \label{thmsigma}
  The map \( \sigma \colon Q \to \mathcal{R} \) defined at
  Definition~\ref{defnsigma} is injective.
\end{theorem}
\begin{remark}
  \label{rem4.12}
  It follows from Remark~\ref{remsigmatilde} that if \( q \in Q \) then \(
  \sigma(q) \) determines the image of \( q \) under the hyperk\"ahler
  moment maps for the actions of \( T \) and \( K \) on \( Q
  \). Recall that the hyperk\"ahler reductions by the action of \( T
  \) on the universal hyperk\"ahler implosion \( Q \) for \( K=\SU(n)
  \) are closures of coadjoint orbits of \( K_\C = \SL(n,\C) \). In
  particular the hyperk\"ahler reduction at level 0 is the nilpotent
  cone for \( K_\C \), which is identified in \cite{KS} with the
  hyperk\"ahler quotient \( M\hkq \tilde{H} \), where \( \tilde{H} =
  \prod_{k=1}^{n-1} U(k) \). This hyperk\"ahler quotient carries an \(
  \SU(n) \) action induced from the action of this group on the top
  space \( \C^n \) of the quiver.  We recalled in
  \S \ref {sec:nilpotent-cone} that, for any choices of complex
  structures, the complex moment map \( M\hkq \tilde{H} \to \lie k _\C
  \) for this action induces a bijection from \( M\hkq \tilde{H} \)
  onto the nilpotent cone in \( \lie k _\C \), and thus the
  hyperk\"ahler moment map \( M\hkq \tilde{H} \to \lie k \otimes \R^3
  \) provides a bijection from \( M\hkq \tilde{H} \) to its image in
  \( \lie k \otimes \R^3 \). Moreover this image is a \( K \times
  \SU(2) \)-invariant subset \( \Nil(K) \) of \( \lie k \otimes \R^3
  \) such that after acting by any element of \( \SU(2) \) the
  projection \( \Nil(K) \to \lie k _\C \) given by the decomposition
  \( \R^3 = \C \oplus \R \) is a bijection onto the nilpotent cone in
  \( \lie k _\C \).
\end{remark}

We obtain an induced map \( \tilde{\sigma} \) from the twistor space
\( \twist_Q \) of \( Q \) to the vector bundle
\begin{equation*}
  \cO(2) \otimes (\lie k_\C \oplus \lie t_\C) \oplus 
  \bigoplus_{j=1}^{n-1}  \cO(\ell_j) \otimes \wedge^j\C^n
\end{equation*}
over \( \PP^1 \). It is the composition of the product of the identity
on \( \PP^1 \) and \( \sigma \) from \( Q \) to \( \mathcal{R} \)
with the natural evaluation map from
\(   \PP^1 \times \mathcal{R} \)  
 to
\begin{equation*}
  \cO(2) \otimes (\lie k_\C \oplus \lie t_\C) \oplus 
  \bigoplus_{j-1}^{n-1}  \cO(\ell_j) \otimes \wedge^j\C^n.
\end{equation*}
As in Remark~\ref{remsigmatilde} we see that \( \tilde{\sigma} \) is
holomorphic and \( K \times T \times \SU(2) \)-equivariant.  We will
prove in \S \ref {sec:embeddings}:

\begin{theorem}
  \label{thmsigmatilde}
  The map
  \begin{equation*}
    \tilde{\sigma} \colon  \twist_Q \to
    \cO(2) \otimes (\lie k_\C \oplus \lie t_\C) \oplus 
    \bigoplus_{j=1}^{n-1}  \cO(\ell_j) \otimes \wedge^j\C^n
  \end{equation*}
  is generically injective; that is, its restriction to a dense
  Zariski-open subset of \( \twist_Q \) is injective.
\end{theorem}
\begin{remark}
  It follows from Remark~\ref{kcbyn} and \eqref{eqnkappa} below that the
  image of \( \tilde{\sigma} \) is contained in the subvariety of \(
  \cO(2) \otimes (\lie k_\C \oplus \lie t_\C) \oplus
  \bigoplus_{j=1}^{n-1} \cO(\ell_j) \otimes \wedge^j\C^n \) whose fibre at
  any \( p \in \PP^1 \) is identified (after choosing any basis vector
  for \( \cO(1)_p \) and thus for \( \cO(\ell_j)_p \) for all \( j
  \geqslant 1 \)) with the product of
  \begin{equation*}
    \{\, (\eta, \xi) \in \lie k_\C \times \lie t_\C: \text{\( \eta \)
    and \( \xi \) have the same eigenvalues}\, \} 
  \end{equation*}
  and the canonical affine completion \( K_\C \symp N \) of \( K_\C/N
  \). We can also impose the condition provided by Lemma~\ref{DKSlem5.9}.
  The image of \( \sigma \) satisfies analogous constraints.
\end{remark}

\section{Stratifying the universal hyperk\"ahler implosion for \(
\SU(n) \) and its twistor space}
\label{sec:strat-univ-hyperk}

In this section we recall the stratification given in \cite{DKS} of
the universal hyperk\"ahler implosion \( Q \) for \( K=\SU(n) \) into
strata which are hyperk\"ahler manifolds, and its refinement in
\cite{DKS2}. The refined stratification has strata \( Q_{[\sim,\cO]}
\) indexed in terms of Levi subgroups and nilpotent orbits in \(
K_\C=\SL(n, \C) \). The latter stratification is not hyperk\"ahler but
reflects well the group structure of \( K=\SU(n) \).  These
stratifications induce corresponding stratifications of the twistor
space of \( Q \).

First of all, given a quiver we may decompose each space in the quiver
into generalised eigenspaces \( \ker (\alpha_i \beta_i - \tau I)^m \)
of \( \alpha_i \beta_i \).  We showed in \cite{DKS} using the complex
moment map equations \eqref{eq:mmcomplex}, that \( \beta_i \) and \(
\alpha_i \) preserve this decomposition.  More precisely, we have
\begin{equation}
  \label{betai}
  \beta_i \colon \ker (\alpha_i \beta_i - \tau I)^m \rightarrow
  \ker(\alpha_{i-1} \beta_{i-1} - (\lambda_i^\C + \tau)I)^m. 
\end{equation}
and
\begin{equation}
  \label{alphai}
  \alpha_i \colon \ker (\alpha_{i-1} \beta_{i-1} - (\lambda_i^\C +
  \tau) I)^m \rightarrow \ker (\alpha_i \beta_i - \tau I)^m. 
\end{equation}
So we actually have a decomposition into subquivers.  Moreover we
showed the maps \eqref{betai} and \eqref{alphai} are bijective unless
\( \tau = 0 \).

It follows that \( \tau \neq 0 \) is an eigenvalue of \( \alpha_i
\beta_i \) if and only if \( \tau + \lambda_i^\C \neq \lambda_i^\C \)
is an eigenvalue of \( \alpha_{i-1} \beta_{i-1} \). Moreover \(
\alpha_i \beta_i \) has zero as an eigenvalue and \( \alpha_i, \beta_i
\) restrict to maps between the associated generalised eigenspace with
eigenvalue \( 0 \) and the generalised eigenspace for \( \alpha_{i-1}
\beta_{i-1} \) associated to \( \lambda_i^\C \) (which could be the
zero space).

One can deduce that the trace-free part \( X^0 \) of \( X= \alpha_{n-1}
\beta_{n-1} \) now has eigenvalues \( \kappa_1,\dots,\kappa_n \),
where
\begin{equation*}
  \begin{split}
    \kappa_j &= \frac1n\Bigl( \lambda^\C_1 + 2 \lambda_2^\C + \dots +
    ({j-1}) \lambda^\C_{j-1} \eqbreak[4] -(n-j) \lambda_j^\C -
    (n-{j-1})\lambda_{j+1}^\C - \dots - \lambda_{n-1}^\C\Bigr).
  \end{split}
\end{equation*} 
In particular if \( i<j \) then
\begin{equation} \label{eqnkappa} \kappa_j - \kappa_i = \lambda_i^\C +
  \lambda_{i+1}^\C + \dots + \lambda_{j-1}^\C .
\end{equation}
This shows that to understand the quiver it is important to understand
when collections of \( \lambda_i \) sum to zero.  

Now we recall that
\begin{equation*}
  T = (S^1)^{n-1} = \prod_{k=1}^{n-1} \Un(k)/\SU(k) = \tilde{H}/H
\end{equation*}
acts on \( Q = M \hkq H \) with hyperk\"ahler moment map
\begin{equation*}
  \mu_{(S^1)^{n-1}}\colon Q \to \tf \otimes \R^3 = (\R^3)^{n-1} = 
  (\C \oplus \R)^{n-1}
\end{equation*}
which maps a quiver to
\begin{equation*}
  (\lambda_1, \ldots, \lambda_{n-1})= (\lambda_1^\C,\lambda_1^\R, \ldots, \lambda_{n-1}^\C,\lambda_{n-1}^\R).
\end{equation*}

\begin{definition}  
  For each choice of \( (\lambda_1, \ldots, \lambda_{n-1}) \) we
  define an equivalence relation \( \sim \) on \( \{1,\dots,n\} \) by
  declaring that if \( 1 \leqslant i<j \leqslant n \) then
  \begin{equation*}
    i \sim j \iff \sum_{k=i}^{j-1} \lambda_k = 0
    \ \text{in}\ \R^3 .
  \end{equation*}
  There is thus a stratification of \( (\R^3)^{n-1}=\tf \otimes \R^3
  \) into strata \( (\R^3)^{n-1}_{\sim} = (\tf \otimes \R^3)_\sim \),
  indexed by the set of equivalence relations \( \sim \) on \(
  \{1,\dots,n\} \), where
  \begin{gather*}
    (\R^3)^{n-1}_{\sim} = \{ (\lambda_1, \ldots, \lambda_{n-1}) \in
    (\R^3)^{n-1}:
    \text{if}\ 1 \leqslant i<j \leqslant n\ \text{then}\\
    i \sim j \iff \sum_{k=i}^{j-1} \lambda_k = 0 \ \text{in}\ \R^3\}.
  \end{gather*}
\end{definition}

Under the identification of \( T \) with \( (S^1)^{n-1} \) using the
positive simple roots as a basis for \( \lie t \) this stratification
of \( (\R^3)^{n-1}= \tf \otimes \R^3 \) is induced by the
stratification of \( \tf \) associated to the root planes in \( \tf \)
(see \cite[\S3]{DKS} and Remark~\ref{hypertoricT}).

We thus obtain a stratification of \( Q \) into subsets \( Q_{\sim}
\), which are the preimage in \( Q \) under \( \mu_{(S^1)^{n-1}} \) of
\( (\R^3)^{n-1}_{\sim} \).

The choice of \( \sim \) corresponds to the choice of a subgroup \(
K_\sim \) of \( K \) which is the compact real form of a Levi subgroup
of \( K_\C \); this subgroup \( K_\sim \) is the centraliser of \(
\mu_{(S^1)^{n-1}}(\q) \in \tf \otimes \R^3 \) for any \( \q \in
Q_{\sim} \).

We observe from \eqref{eqnkappa} that if \( i \sim j \) then we have
equality of the eigenvalues \( \kappa_i \) and \( \kappa_j \).

Now let \( Q^\circ \) denote the subset of \( Q \) consisting of
quivers such that
\begin{equation*}
  \sum_{k=i}^{j-1} \lambda_k = 0
  \ \text{in}\ \R^3 \iff \sum_{k=i}^{j-1} \lambda_k^\C = 0
  \ \text{in}\ \C,
\end{equation*}
and for each equivalence relation \( \sim \) on \( \{1,2,\ldots, n\}
\) let \( Q^\circ_\sim \) denote its intersection with \( Q_\sim \),
consisting of quivers such that
\begin{equation*}
  i \sim j \iff
  \sum_{k=i}^{j-1} \lambda_k = 0
  \ \text{in}\ \R^3 \iff \sum_{k=i}^{j-1} \lambda_k^\C = 0
  \ \text{in}\ \C.
\end{equation*}
The full implosion \( Q \) is the sweep of \( Q^\circ \) under the \(
SU(2) \) action.

For a quiver \( \q \) in \( Q^\circ \) the equivalence relation \(
\sim \) for which \( \q \in Q_\sim \) is determined by the fact that
we have equality of the eigenvalues \( \kappa_i \) and \( \kappa_j \)
of the trace-free part \( X^0 \) of \( X= \alpha_{n-1} \beta_{n-1} \) if
and only if \( i \sim j \). In particular if \( \q\in Q^\circ \) then
\( (K_\sim)_\C \) is the subgroup of \( K_\C \) which preserves the
decomposition of \( \q \) into the subquivers determined by the
generalised eigenspaces of the compositions \( \alpha_i\beta_i \).

\begin{remark}
  Unfortunately \( Q^\circ \) is not an open subset of \( Q \),
  although its intersection \( Q^\circ_\sim \) with \( Q_\sim \) is
  open in \( Q_\sim \) for each \( \sim \). In fact we will show in
  forthcoming work that there is a desingularisation \( \hat{Q} \) of
  \( Q \) covered by open subsets \( s \hat{Q}^\circ \) for \( s \in
  \SU(2) \) such that the image of the open subset \( \hat{Q}^\circ \)
  of \( \hat{Q} \) under \( \hat{Q} \to Q \) is \( Q^\circ
  \).
\end{remark}

Let us now return to considering the decomposition of the quiver into
subquivers
\begin{equation*} \cdots V_i^j
  \stackrel[\beta_{i,j}]{\alpha_{i,j}}{\rightleftarrows} V_{i+1}^j
  \cdots
\end{equation*}
determined by the generalised eigenspaces (with eigenvalues \(
\tau_{i+1, j} \)) of the compositions \( \alpha_i\beta_i \), such that
\begin{equation*}
  \alpha_{i,j} \beta_{i,j} - \beta_{i+1,j} \alpha_{i+1,
  j} = \lambda_{i+1}^\C
\end{equation*}
and \( \alpha_{i,j} \) and \( \beta_{i,j} \) are isomorphisms unless
\( \tau_{i+1, j}=0 \).  If for some \( j \) we have that \(
\alpha_{k,j}, \beta_{k,j} \) are isomorphisms for \( i+1 \leqslant k <
s \) but not for \( k=i,s \), then it follows that \( \tau_{i+1,j} =
\tau_{s+1,j}=0 \), hence \( \sum_{k=i+1}^s \lambda_k^\C =0 \), and so
since the quiver lies in \( Q_{\sim}^\circ \) we have
\begin{equation*}
  \sum_{k=i+1}^s \lambda_k =0 \in \R^3.
\end{equation*}

As explained in \cite{DKS} and \cite{DKS2}, we may contract the
subquivers at edges where the maps are isomorphisms. Explicitly, if \(
\alpha_{i,j} \) and \( \beta_{i,j} \) are isomorphisms (which will
occur when the associated \( \tau_{i+1,j} \) is non-zero), then we may
replace
\begin{equation*}
  \cdots  V_{i-1}^j
  \stackrel[\beta_{i-1,j}]{\alpha_{i-1,j}}{\rightleftarrows}
  V_i^j
  \stackrel[\beta_{i,j}]{\alpha_{i,j}}{\rightleftarrows}
  V_{i+1}^j  
  \stackrel[\beta_{i+1,j}]{\alpha_{i+1,j}}{\rightleftarrows}
  V_{i+2}^j
\end{equation*}
with
\begin{equation*}
  V_{i-1}^j
  \stackrel[\beta_{i-1,j}]{\alpha_{i-1,j}}{\rightleftarrows}
  V_i^j
  \stackrel[(\alpha_{i,j})^{-1}\beta_{i+1,j}]{\alpha_{i+1,j}  \alpha_{i,j}}
  {\rightleftarrows} V_{i+2}^j,
\end{equation*}
and then the complex moment map equations are satisfied with
\begin{equation*}
  \alpha_{i-1,j}\beta_{i-1,j} - (\alpha_{i,j})^{-1}\beta_{i+1,j} \alpha_{i+1,j} 
  \alpha_{i,j} = \lambda_{i-1}^\C + \lambda_i^\C.
\end{equation*}
If we fix an identification of \( V_{i+1}^j \) with \( V_i^j \) and
apply the action of \( \SL(V_{i,j}) \) so that \( \alpha_{i,j} \) is a
non-zero scalar multiple \( aI \) of the identity, then \( \beta_{i,j}
\) is determined by \( \alpha_{i-1,j}, \alpha_{i+1,j}, \beta_{i-1,j},
\beta_{i+1,j} \) and the scalars \( a \) and \( \lambda_i^\C \) via
the equations \eqref{eq:mmcomplex} (see \cite{DKS} for more details).

After performing such contractions, the resulting quivers satisfy the
complex moment map equations with zero scalars. In other words, they
satisfy the complex moment map equations for the product of the
relevant \( \GL(V_i^j) \). In fact, because our quiver is in \(
Q^\circ \), the full hyperk\"ahler moment map equations for the
associated product of unitary groups are satisfied, and the orbits
under the action of the complex group are closed.

\begin{remark}
  \label{contract}
  We are now in the situation analysed by the third author and
  Kobak~\cite{KS} in their construction of the nilpotent variety, as
  discussed in \S \ref {sec:nilpotent-cone}. Their results, in particular
  Theorem 2.1 (cf.\ \cite{DKS} Proposition 5.16), show that each
  contracted subquiver is the direct sum of a quiver where all \(
  \alpha \) are injective and all \( \beta \) are surjective and a
  quiver in which all maps are \( 0 \). Moreover the direct sum of the
  contracted subquivers is completely determined (modulo the action of
  the product of the \( \GL \) groups) by the elements \( \alpha_{n-1}
  \beta_{n-1} \) at the top edge of each injective/surjective
  subquiver.  The argument of \cite{KS} shows these are actually
  nilpotent.

  We also observe that because \( \sim \) determines the decomposition
  of the original quiver into eigenspaces, the direct sum of these
  nilpotents is actually a nilpotent element of \( (\lie k _\sim)_\C
  \). It coincides with \( X^0_n \), the nilpotent part in the Jordan
  decomposition of the trace-free part \( X^0 \in \lie k_\C \) of \( X
  = \alpha_{n-1}\beta_{n-1} \).  (Recall that this is the unique
  decomposition \( X^0 = X^0_s + X^0_n \) where \( X^0_s \) and \(
  X^0_n \) in \( \lie k_\C \) satisfy \( [X^0_s,X^0_n] = 0 \) and \(
  X^0_s \) is semisimple while \( X^0_n \) is nilpotent).
  Furthermore, given \( \sim \), the adjoint orbit of this nilpotent
  element in \( (\lie k _\sim)_\C \) corresponds precisely to
  determining the dimensions of the various vector spaces in the
  injective/surjective subquivers (see \cite{DKS2} Remarks 5.10 and
  5.11).  For example, if a quiver has all \( \lambda_i=0 \), then \(
  \sim \) has a single equivalence class, \( \lie k _\sim = \lie k \),
  and the choice of \( \cO \) is just the choice of a nilpotent orbit
  in \( \lie k _{\C} \).  At the other extreme, if no non-trivial sums
  are zero, the equivalence classes are singletons and \( K_\sim \) is
  a torus. The orbit \( \cO \) must now be zero.
\end{remark}

To each quiver in \( Q^\circ \) we have associated an equivalence
relation \( \sim \) and a nilpotent orbit \( \cO \) in \( (\lie k
_\sim)_\C \).  Let \( Q_{[\sim, \cO]}^\circ \) denote the set of
quivers in \( Q^\circ \) with given \( \sim \) and \( \cO \), and let
\( Q_{[\sim, \cO]} \) denote the \( SU(2) \) sweep of \( Q_{[\sim,
\cO]}^\circ \).  We may therefore stratify \( Q \) as a disjoint union
\begin{equation*} Q = \coprod_{\sim, \cO} Q_{[\sim,\cO]}
\end{equation*}
over all equivalence relations \( \sim \) on \( \{1,\ldots,n\} \) and
all nilpotent adjoint orbits \( \cO \) in \( (\lie k _\sim)_\C \).
\begin{remark}
  Defining \( Q_{[\sim, \cO]} \) as the \( SU(2) \) sweep of \(
  Q_{[\sim, \cO]}^\circ \) in this way, it is not clear that the
  strata \( Q_{[\sim, \cO]} \) are disjoint. Hence in \cite{DKS2} a
  different approach is taken in which the stratification \( \{
  Q_{[\sim, \cO]} \} \) of \( Q \) is initially indexed differently;
  the equivalence between the two viewpoints is made in \cite{DKS2}
  Remark 5.13.
\end{remark}
\begin{remark} \label{detstratum} The stratum \( Q_{[\sim, \cO]} \) in
  which a quiver lies is determined by the values at the quiver of the
  hyperk\"ahler moment maps for the actions on \( Q \) of \( K =
  \SU(n) \) and \( T = (S^1)^{n-1} \).

  For the value \( (\lambda_1, \dots, \lambda_{n-1}) \) of \(
  \mu_{(S^1)^{n-1}} \) determines the equivalence relation \( \sim \)
  and also the generic choices of complex structures for which
  \begin{equation} \label{star} \sum_{k=i}^{j-1} \lambda_k = 0 \
    \text{in}\ \R^3 \iff \sum_{k=i}^{j-1} \lambda_k^\C = 0 \
    \text{in}\ \C.
  \end{equation}
  Moreover for such choices of complex structures the quiver
  decomposes as a direct sum of subquivers determined by the
  generalised eigenspaces of the composition \(
  \alpha_{n-1}\beta_{n-1} \), and this is given by the complex moment
  map for the action of \( K \).  It follows that the Jordan type of
  \( \alpha_{n-1}\beta_{n-1} \) (for one of the generic choices of
  complex structures for which \eqref{star} holds) determines the
  nilpotent orbit \( \cO \) in \( (\lie k_\sim)_\C \).
\end{remark}

\begin{remark}
  The stratification of \( Q \) into strata \( Q_{[\sim, \cO]} \)
  induces a stratification of the twistor space \( \twist_Q \) into
  strata \( (\twist_Q)_{[\sim, \cO]} \).
\end{remark}

Let \( \sim \) be an equivalence relation on \( \{1, \ldots, n\} \),
and let \( \cO \) be a nilpotent adjoint orbit in \( (\lie k _\sim)_\C
\).  In \cite{DKS2} we explained how quivers in \( Q_{[\sim, \cO]}^\circ
\) may be put in standard forms using Jordan canonical form.

As above we first decompose \( \q \) into a direct sum of subquivers
determined by the generalised eigenspaces of the compositions \(
\alpha_i\beta_i \).  Since \( \q \) lies in \( Q_{[\sim,\cO]}^\circ \)
each such subquiver is the direct sum of a quiver \( \q^{[j]} \) of
the form
\begin{equation} \label{eeq6.8} 0
  \stackrel[\beta^{[j]}_0]{\alpha^{[j]}_0}{\rightleftarrows}
  \C^{m_1}\stackrel[\beta_1^{[j]}]{\alpha_1^{[j]}}{\rightleftarrows}
  \C^{m_2}\stackrel[\beta_2^{[j]}]{\alpha_2^{[j]}}{\rightleftarrows}\dots
  \stackrel[\beta_{n-2}^{[j]}]{\alpha_{n-2}^{[j]}}{\rightleftarrows}
  \C^{m_{n-1}}
  \stackrel[\beta_{n-1}^{[j]}]{\alpha_{n-1}^{[j]}}{\rightleftarrows}
  \C^{m_n}
\end{equation}
where the maps \( \alpha_k^{[j]} \) for \( 1\leqslant k \leqslant n-1
\) are injective and the maps \( \beta_k^{[j]} \) for \( 1\leqslant k
\leqslant n-1 \) are surjective, together with quivers of the form
(for \( 1 \leqslant h \leqslant p \) )
\begin{equation*}
  \C^{d_h}
  \stackrel[\beta_{i_h}^{(h)}]{\alpha_{i_h}^{(h)}}{\rightleftarrows}
  \C^{d_h} \rightleftarrows \dots \rightleftarrows \C^{d_h}
  \stackrel[\beta_{j_h-2}^{(h)}]{\alpha_{j_h-2}^{(h)}}{\rightleftarrows}
  \C^{d_h}    
\end{equation*}
in the places \( i_h,i_h + 1, \dots ,j_h -1 \), where the maps \(
\alpha_k^{(h)} \), \( \beta_k^{(h)} \), for \( i_h \leqslant k < j_h -
1 \), are multiplication by complex scalars such that \(
\gamma_k^{(h)} = \alpha_k^{(h)} + j \beta_k^{(h)} \in \HH\setminus
\{0\} \). Moreover the combinatorial data here and the Jordan type of
\( \alpha_{n-1}\beta_{n-1} \) for each summand \eqref{eeq6.8} is
determined by the pair \( (\sim,\cO) \).

As explained in \cite{DKS2}, we may use complex linear changes of
coordinates in \( K_\C \times H_\C = \prod_{k=1}^n \SL(k,\C) \) to put
\( \alpha_{n-1}\beta_{n-1} \) into Jordan canonical form and then
decompose the quiver \eqref{eeq6.8} into a direct sum of quivers
determined by the Jordan blocks of \(
\alpha^{[j]}_{n-1}\beta^{[j]}_{n-1} \). More precisely, \(
\alpha^{[j]}_k \) is a direct sum over the set \( B_j \) of Jordan
blocks for \( \alpha^{[j]}_{n-1}\beta^{[j]}_{n-1} \) of matrices of
the form
\begin{equation} \label{formstar}
  \begin{pmatrix}
    \xi_1^{bjk} & 0 & \cdots & 0 & 0  \\
    \nu_1^{bjk} & \xi_2^{bjk} & 0 & \cdots & 0\\
    0 & \nu_2^{bjk} &  & \cdots & 0\\
    & & \cdots & & \\
 0 &  \cdots & 0 & \nu^{bjk}_{\ell_b - n + k -1}
    & \xi^{bjk}_{\ell_b - n +k }  \\
    0 & \cdots & 0 & 0 & \nu^{bjk}_{\ell_b - n+k}
  \end{pmatrix} \end{equation} for some \( \nu^{bjk}_i, \xi_i^{bjk}
\in\C^* \) where \( \ell_b \) is the size of the Jordan block \( b \in
B_j \), while \( \beta_k^{[j]} \) is a corresponding direct sum over
\( b \in B_j \) of matrices of the form
\begin{equation} \label{formstarbeta}
  \begin{pmatrix}
    0 & \mu_1^{bjk} & 0
    & 0 & \cdots & 0\\
    0 & 0 & \mu_2^{bjk} & 0
    & \cdots & 0\\
    & & \cdots & & & \\
    0 & 0 & \cdots & 0 & \mu^{bjk}_{\ell_b - n + k -1}
    &
    0 \\
    0 & 0 & \cdots & 0 & 0 & \mu^{bjk}_{\ell_b -
    n+k}
  \end{pmatrix}
\end{equation}
for some \( \mu_i^{bjk} \in\C^* \), all satisfying the complex moment
map equations \eqref{eq:mmcomplex}. The quiver given by the direct sum
over all the Jordan blocks \( \bigcup_{j}B_j \) for \(
\alpha_{n-1}\beta_{n-1} \) has closed \( (H_S)_\C \)-orbit.  If we
allow complex linear changes of coordinates in \( K_\C \times
\tilde{H}_\C = \SL(n, \C) \times \prod_{k=1}^{n-1} \GL(k,\C) \) (or
equivalently allow the action of its quotient group \( K_\C \times
T_\C \) on \( Q_{[\sim,\cO]} \)), then the quiver can be put into a
more restricted form which is completely determined by \(
\alpha_{n-1}\beta_{n-1} \) and \( (\lambda_1^\C, \ldots,
\lambda_{n-1}^\C) \), and hence by the value of the complex moment map
for the action of \( K \times T \) on \( Q \).

\begin{remark}
  Let \( \{ e_1, \ldots, e_n\} \) be the standard basis for \( \C^n
  \). When complex linear changes of coordinates in \( K_\C \times H_\C
  \) are used to put the quiver in the standard form given by
\eqref{formstar}, \eqref{formstarbeta} then
  \begin{equation*}
    \wedge^j \alpha_{n-1} \circ \cdots \circ \alpha_j
  \end{equation*}
  takes the standard basis vector for \( \wedge^j \C^j \) to a scalar
  multiple of \( e_1 \wedge \cdots \wedge e_j \in \wedge^j \C^n \).
\end{remark}

Let \( Q_{[\sim,\cO]}^{\circ,JCF} \) be the subset of \(
Q_{[\sim,\cO]}^{\circ} \) representing quivers of the standard form
described above via \eqref{formstar}, \eqref{formstarbeta} where \(
\alpha_{n-1}\beta_{n-1} \) is in Jordan canonical form and the
summands of the quiver corresponding to generalised eigenspaces of the
compositions \( \alpha_i \beta_i \) (and thus to equivalence classes
for~\( \sim \)) are ordered according to the usual ordering on the
minimal elements of the equivalence classes, and the Jordan blocks for
each equivalence class are ordered by size. Then in particular we have
\begin{equation*}
  Q_{[\sim,\cO]}^\circ = K_\C Q_{[\sim,\cO]}^{\circ,JCF}
\end{equation*}
and the nonempty fibres of the complex moment map
\begin{equation*}
  Q_{[\sim,\cO]}^{\circ} \to \lie k _\C \oplus \tf_\C
\end{equation*}
for the action of \( K \times T \) are contained in \( K_\C \times
T_\C \)-orbits (see \cite{DKS2} \S7 and in particular Lemma 7.5).
 
\begin{remark} \label{remqt} We can also identify \(
  Q_{[\sim,\cO]}^{\circ,JCF} \) with an open subset of a hypertoric
  variety by replacing all the \( \xi \) entries in
  \eqref{formstarbeta} with zero (see \cite{DKS2} Lemma 7.13).
  
  More precisely, if a quiver \( \q \) has \( \alpha, \beta \) maps of
  the form given by \eqref{formstar}, \eqref{formstarbeta}, then we
  may obtain a new quiver which still satisfies the complex moment map
  equations by replacing all the \( \xi \) entries in \( \beta \) by
  zero. The resulting maps are denoted by \( \alpha^T, \beta^T \).

  So if \( \q \) is any quiver representing a point in \(
  Q_{[\sim,\cO]}^{\circ,JCF} \) whose Jordan blocks are of the form
  given by \( \alpha_k \) and \( \beta_k \) as above, then we may form
  a new quiver \( \q^T \) from \( \q \) by replacing each such Jordan
  block with the quiver given by \( \alpha_k^T \) and \( \beta_k^T
  \). The new quiver now satisfies the complex moment map equations
  for the action of \( H \), or equivalently for the action of the
  maximal torus \( T_H \) of \( H \).
\end{remark}

\begin{remark}
  \label{remchoice}
  The subgroup of \( K_\C \times T_\C \) preserving the standard form
  must preserve the decomposition of \( \q \) into subquivers given by
  the generalised eigenspaces and hence must lie in \( (K_\sim)_\C
  \times T_\C \).

  Let \( P \) be the parabolic subgroup of \( (K_\sim)_\C \) which is
  the Jacobson--Morozov parabolic of the element of the nilpotent orbit
  \( \cO \) for \( (K_\sim)_\C \) given by the nilpotent component of
  \( X^0 \).  In \cite{DKS2} we identified the group which preserves
  the standard form as \( R_{[\sim,\cO]} \times T_\C \) where \(
  R_{[\sim,\cO]} \) is the centraliser in \( P \) of this nilpotent
  element.  It follows that
  \begin{equation} \label{identity} Q_{[\sim,\cO]}^\circ \cong (K_\C
    \times T_\C) \times_{(R_{[\sim,\cO]} \times T_\C) }
    Q_{[\sim,\cO]}^{\circ,JCF} \cong K_\C \times_{R_{[\sim,\cO]} }
    Q_{[\sim,\cO]}^{\circ,JCF}. \end{equation}
  Moreover \( [P,P] \cap R_{[\sim,\cO]} \) acts trivially on \(
  Q_{[\sim,\cO]}^{\circ,JCF} \). If we define \( T_{[\sim,\cO]} \) to
  be the intersection of \( R_{[\sim,\cO]} \) with the maximal torus
  \( T \), then \( (T_{[\sim,\cO]})_\C /[P,P] \cap (T_{[\sim,\cO]})_\C
  \) acts freely on \( Q_{[\sim,\cO]}^{\circ,JCF} \).
\end{remark}

The situation is summarised in the following theorem which is
\cite{DKS2} Theorem 8.1.

\begin{theorem}
  \label{thm6.8}
  For each equivalence relation \( \sim \) on \( \{1,\ldots,n\} \) and
  nilpotent adjoint orbit \( \cO \) for \( (K_\sim)_\C \), the stratum
  \( Q_{[\sim,\cO]} \) is the union over \( s \in \SU(2) \) of its
  open subsets \( sQ_{[\sim,\cO]}^\circ \), and
  \begin{equation*}
    Q_{[\sim,\cO]}^\circ \cong  K_\C \times_{R_{[\sim,\cO]}  }
    Q_{[\sim,\cO]}^{\circ,JCF}
  \end{equation*}
  where \( R_{[\sim,\cO]} \) is the centraliser in \( (K_\sim)_\C \)
  of the standard representative \( \xi_0 \) in Jordan canonical form
  of the nilpotent orbit \( \cO \) in \( (\lie k _\sim)_\C \), and \(
  Q_{[\sim,\cO]}^{\circ,JCF} \) can be identified with an open subset
  of a hypertoric variety. The image of the restriction
  \begin{equation*}
    Q_{[\sim,\cO]}^{\circ} \to \lie k _\C
  \end{equation*}
  of the complex moment map for the action of \( K \) on \( Q \) is \(
  K_\C( (\tf_\C)_\sim \oplus \cO) \cong K_\C \times_{(K_\sim)_\C} (
  (\tf_\C)_\sim \oplus \cO) \) and its fibres are single \(
  (T_{[\sim,\cO]})_\C \times T_\C \)-orbits, where \(
  (T_{[\sim,\cO]})_\C = T_\C \cap R_{[\sim,\cO]} \) and \(
  (T_{[\sim,\cO]})_\C /[P,P] \cap (T_{[\sim,\cO]})_\C \) acts freely
  on \( Q_{[\sim,\cO]}^{\circ,JCF} \). Here \( P \) is the
  Jacobson--Morozov parabolic of an element of the nilpotent orbit \(
  \cO \) for \( (K_\sim)_\C \), and \( [P,P] \cap (T_{[\sim,\cO]})_\C
  \) acts trivially on \( Q_{[\sim,\cO]}^{\circ,JCF} \).
\end{theorem}

\section{Embeddings}
\label{sec:embeddings}

In this section we will prove that the map \( \sigma \) defined in
\S \ref {sec:towards-an-embedding} is injective and the map \(
\tilde{\sigma} \) defined in \S \ref {sec:towards-an-embedding} is
generically injective.

First consider the restriction of \( \sigma \) to the image \( Q_T \)
in \( Q \) of the hypertoric variety \( M_T \hkq T_H \) under the
natural map \( \iota\colon M_T \hkq T_H \to Q \) (see
Definition~\ref{defn3.10}).

\begin{lemma} \label{hypertoricsigma} The restriction of
  \begin{equation*}
    \sigma\colon Q  \to \mathcal{R} =
    H^0(\PP^1,
    (\cO(2) \otimes (\lie k_\C \oplus  \lie t_\C)) \oplus 
    \bigoplus_{j=1}^{n-1}  \cO(\ell_j) \otimes \wedge^j\C^n)
  \end{equation*}
  to \( Q_T = \iota(M_T \hkq T_H) \) is injective.
\end{lemma}

This follows immediately from

\begin{lemma} \label{sigmat} The restriction to \( Q_T = \iota(M_T
  \hkq T_H) \) of the projection
  \begin{equation*}
    \sigma_T\colon Q \to 
    H^0(\PP^1, (\cO(2)
    \otimes \lie t_\C) \oplus 
    \bigoplus_{j=1}^{n-1}  \cO(\ell_j) \otimes \wedge^j\C^n)
  \end{equation*}
  of \( \sigma \) is injective.
\end{lemma}
\begin{proof}
  By Remark~\ref{remsigmatilde} we can recover the value of the
  hyperk\"ahler moment map for the action of \( T \) on \( M_T\hkq T_H
  \) at any point from its image under \( \sigma_T \circ \iota \).
  Since \( M_T\hkq T_H \) is hypertoric the fibres of this
  hyperk\"ahler moment map are \( T \)-orbits, so it suffices to show
  that \( \sigma_T \) is injective on \( T \)-orbits in \( Q_T \).

  Recall that for any \( t \in T \) the action on \( Q_T \) of \(
  (t,1) \in K \times T \) is the same as the action of \( (1,t) \in K
  \times T \). If \( \q \in Q_T \) then \( \q \in s
  Q^\circ_{[\sim,\cO]} \) for some \( s \in \SU(2) \) and stratum \(
  Q^\circ_{[\sim,\cO]} \) with the nilpotent orbit \( \cO \) equal to
  the zero orbit \( \{0\} \) in \( \lie k_\C^* \). Hence in the
  notation of Theorem~\ref{thm6.8} we have
  \begin{equation*}
    R_{[\sim, \cO]} = P = (K_\sim)_\C \quad\text{and}\quad T_{[\sim, \cO]} = T,
  \end{equation*}
  and thus the stabiliser of \( \q \) in \( T \) is
  \begin{equation*}
    T \cap [P,P] = T \cap [K_\sim,K_\sim] = T \cap [K_\lambda,K_\lambda]
  \end{equation*}
  where \( \lambda \) is the image of \( \q \) under the hyperk\"ahler
  moment map for the action of \( T \).  Both \( \sigma_T \) and \(
  \iota \) are \( T \)-equivariant, so it is enough to show that the
  stabiliser in \( T \) of \( \sigma_T(\q) \) is contained in \( T
  \cap [K_\lambda,K_\lambda] \). But it follows from Remark~\ref{kcbyn} that
  the stabiliser of \( \sigma_T(\q) \) in \( K \) is contained in
  \begin{equation*}
    K_\lambda \cap (K \cap [P_\lambda, P_\lambda]) = [K_\lambda, K_\lambda],
  \end{equation*}
  where \( P_\lambda \) is the standard parabolic in \( K_\C \) whose Levi
  subgroup is \( (K_\lambda)_\C \) and whose intersection with \( K \)
  is \( K_\lambda \), so the result follows.  In more detail, consider
  a quiver \( \q \) in \( M_T \) given as in Definition~\ref{defn3.10} by
  \begin{equation*}
    \alpha_k = \begin{pmatrix}
      0 & \cdots  & 0 & 0\\
      \nu_1^k &  0 & \cdots & 0\\
      0 & \nu_2^k  & \cdots & 0\\
      & & \cdots  & \\
      0 & \cdots  & 0 & \nu_k^k
    \end{pmatrix}
  \end{equation*}
  and
  \begin{equation*}
    \beta_k = \begin{pmatrix}
      0 & \mu_1^k & 0   & \cdots & 0 \\
      0 & 0 & \mu_2^k  &  \cdots & 0 \\
      & & \cdots  & &  \\
      0 & 0 & \cdots  & 0 & \mu_k^k   \end{pmatrix}
  \end{equation*}
  for some \( \nu^k_i, \mu_i^k \in \C \). For every \( (u,v) \in \C^2
  \) and every \( j \in \{1, \ldots, n-1 \} \) its image under the
  composition \( \sigma_T \circ \iota \) determines the element
  \begin{equation*}
    \wedge^j( u\alpha_{n-1}+ v \beta^*_{n-1})(u \alpha_{n-2}+ v
    \beta_{n-2}^*) \cdots (u\alpha_j 
    + v\beta_j^*) \in \wedge^j \C^n,
  \end{equation*}
  where
  \begin{equation*}
    u\alpha_k + v \beta_k^* = \begin{pmatrix}
      0 & \cdots & 0  & 0 \\
      u\nu_1^k + v \bar{\mu}_1^k &  0 & \cdots & 0\\
      0 & u\nu_2^k + v \bar{\mu}_2^k  & \cdots & 0\\
      & & \cdots  & \\
      0 & \cdots  & 0 & u\nu_k^k + v \bar{\mu}_k^k
    \end{pmatrix},
  \end{equation*}
  and thus determines the product
  \begin{equation*}
    \prod_{k=j}^{n-1} \prod_{i=1}^j (u\nu_i^k + v \bar{\mu}_i^k).
  \end{equation*}
  The action of an element \( t \) of \( T \cong (S^1)^{n-1} \cong
  \prod_{j=1}^{n-1} \Un(j)/\SU(j) \) represented by matrices \( A_j
  \in \Un(j) \) for \( j=1, \ldots, n-1 \) multiplies this product by
  \begin{equation*}
    \prod_{k=j}^{n-1} \det A_k.
  \end{equation*}
  The contracted quivers associated as in Remark~\ref{contract} to a quiver
  of this form are all identically zero. Thus this product is non-zero
  (and hence \( \prod_{k=j}^{n-1} \det A_k = 1 \) if \( t \)
  stabilises \( \sigma_T(\q) \)) precisely when \( j+1 \) is the
  smallest element of its equivalence class under \( \sim \). This
  tells us that the stabiliser in \( T \) of \( \sigma_T(\q) \) is
  contained in \( [K_\lambda, K_\lambda] \) as required.
\end{proof}

\begin{remark}
  \label{remarkst}
  We note that on the subset \( Q_{[\sim,\cO]}^{\circ, JCF} \) embedded in
  the hypertoric variety as in Remark~\ref{remqt}, the fibres of the complex
  moment map for the complex structure associated to \( [1:0] \) are
  \( T_\C \)-orbits. Moreover a straightforward modification of the above
  proof shows that the composition \( (\sigma_T)_{[1:0]} \) of \(
  \sigma_T \) with evaluation at \( [1:0] \in \PP^1 \) is injective on
  \( T_\C \)-orbits in \( Q_{[\sim,\cO]}^{\circ, JCF} \), and thus that
  \( (\sigma_T)_{[1:0]} \) is injective on \( Q_{[\sim,\cO]}^{\circ, JCF} \).
\end{remark}

We now turn to showing injectivity for the map \( \sigma \) on the
full implosion \( Q \).

By Remark~\ref{detstratum} it suffices to show that the restriction of \(
\sigma \) to any stratum \( Q_{[\sim, \cO]} \) is injective.  Indeed,
we need only show that its restriction to \( Q_{[\sim,\cO]}^\circ \)
is injective, since if two points lie in \( Q_{[\sim, \cO]} \) then
there is some \( s \) in \( SU(2) \) such that they both lie in \(
sQ_{[\sim, \cO]}^\circ \), and \( \sigma \) is by construction \(
SU(2) \)-equivariant. We will do this by showing that the restriction
to \( Q_{[\sim,\cO]}^\circ \) of the composition \( \sigma_{[1:0]} \)
of \( \sigma \) with evaluation at \( [1:0] \in \PP^1 \) is injective,
and this will also show that \( \tilde{\sigma} \) is generically
injective.

\begin{proposition}
  \label{proplast}
  For any equivalence relation \( \sim \) on \( \{1,\ldots,n\} \) and
  nilpotent coadjoint orbit \( \cO \) in \( \lie k_\C^* \), the
  restriction
  \begin{equation*}
    \sigma_{[1:0]} \colon Q_{[\sim,\cO]}^\circ \to (\lie k_\C \oplus
    \tf_\C) \otimes \cO(2)_{[1:0]} 
    \oplus \bigoplus_{j=1}^{n-1} \wedge^j(\C^n)  \otimes \cO(\ell_j)_{[1:0]}
  \end{equation*}
  to \( Q_{[\sim,\cO]}^\circ \) of the composition \( \sigma_{[1:0]}
  \) of \( \sigma \) with evaluation at \( [1:0] \in \PP^1 \) is
  injective.
\end{proposition}

\begin{proof}
  \begin{equation*}
    \sigma_{[1:0]} \colon Q_{[\sim,\cO]}^\circ \to (\lie k_\C \oplus
    \lie t_\C) \otimes \cO(2)_{[1:0]} 
    \oplus \bigoplus_{j=1}^{n-1} \wedge^j(\C^n)  \otimes \cO(\ell_j)_{[1:0]}
  \end{equation*}
  is holomorphic and \( K \times T \)-equivariant, so it is \( K_\C
  \times T_\C \)-equivariant. By Theorem~\ref{thm6.8}
  \begin{equation*}
    Q_{[\sim,\cO]}^\circ \cong  K_\C \times_{R_{[\sim,\cO]}  }
    Q_{[\sim,\cO]}^{\circ,JCF}
  \end{equation*}
  where \( R_{[\sim,\cO]} \) is the centraliser in \( (K_\sim)_\C \)
  of the standard representative \( \xi_0 \) in Jordan canonical form
  of the nilpotent orbit \( \cO \) in \( (\lie k _\sim)_\C \).  It
  follows immediately from the definition of \(
  Q_{[\sim,\cO]}^{\circ,JCF} \) by looking at the projection of \(
  \sigma_{[1:0]} \) to \( \lie k_\C \otimes \cO(2)_{[1:0]} \) that if
  \( g \in K_\C \) and \( \q \in Q_{[\sim,\cO]}^{\circ,JCF} \) and
  \begin{equation}\label{botp30} g \sigma_{[1:0]} (\q) \in
    \sigma_{[1:0]} ( Q_{[\sim,\cO]}^{\circ,JCF})
  \end{equation}
  then \( g \in R_{[\sim,\cO]} \).

  Proposition~\ref{proplast} now follows from the following lemma.

  \begin{lemma} \label{lem6.4} The restriction of \( \sigma_{[1:0]} \)
    to \( Q_{[\sim,\cO]}^{\circ,JCF} \) is injective.
  \end{lemma}
  \begin{proof} As in Theorem~\ref{thm6.8}, \( Q_{[\sim,\cO]}^{\circ,JCF} \)
    can be identified with an open subset of a hypertoric variety by
    associating to a quiver \( \q \in Q_{[\sim,\cO]}^{\circ,JCF} \)
    with blocks of the form \eqref{formstar} and \eqref{formstarbeta}
    a quiver \( \q^T \) where each \( \xi^{ijk}_\ell \) in \( \q \) is
    replaced with zero (cf.\ Remark~\ref{remqt}). This hypertoric variety is
    (up to the action of the Weyl group of \( H \times K \)) a
    subvariety of \( Q_T \). Moreover the projection \(
    (\sigma_T)_{[1:0]} \) of \( \sigma_{[1:0]} \) to
    \begin{equation*}
      \lie t_\C \otimes \cO(2)_{[1:0]}
      \oplus \bigoplus_{j=1}^{n-1} \wedge^j(\C^n)  \otimes \cO(\ell_j)_{[1:0]}
    \end{equation*}
    satisfies \( (\sigma_T)_{[1:0]}(\q^T) = (\sigma_T)_{[1:0]}(\q) \),
    so the result follows immediately from Lemma~\ref{sigmat} and
    Remark~\ref{remarkst}.
  \end{proof}

  \sloppy
  To complete the proof of Proposition~\ref{proplast}, suppose that \(
  \sigma_{[1:0]}(g_1 \q_1) = \sigma_{[1:0]}(g_2 \q_2) \) where \( g_1
  \) and \( g_2 \) are elements of \( K_\C \) and \( \q_1 \) and \(
  \q_2 \) are elements of \( Q^{\circ, JCF}_{\sim, \cO} \).  As \(
  \sigma_{[1:0]} \) is equivariant, we have \( r = g_1^{-1} g_2 \in
  R_{[\sim, \cO]} \) as at \eqref{botp30}, and hence \(
  \sigma_{[1:0]}(g_1 \q_1) = \sigma_{[1:0]}(g_1 r \q_2) \).  Hence by
  equivariance \( \sigma_{[1:0]}(\q_1) = \sigma_{[1:0]}(r \q_2) \)
  where both \( r \q_2 \) and \( \q_1 \) lie in \(
  Q_{[\sim,\cO]}^{\circ,JCF} \), so Lemma~\ref{lem6.4} shows that \( \q_1 =
  r \q_2 \) and thus that \( g_1 \q_1 = g_2 \q_2 \) as desired.
\end{proof}

This completes the proof of Theorem~\ref{thmsigma}. Moreover since the map
\begin{equation*}
  \tilde{\sigma} \colon \twist_Q \to 
  \cO(2) \otimes (\lie k_\C \oplus  \lie t_\C) \oplus 
  \bigoplus_{j=1}^{n-1}  \cO(\ell_j) \otimes \wedge^j\C^n
\end{equation*}
is compatible with the projections to \( \PP^1 \) and is \( \SU(2)
\)-equivariant, it follows immediately from Proposition~\ref{proplast} that \(
\tilde{\sigma} \) is injective on the dense subset of \( \twist_Q =
\PP^1 \times Q \) which is the union over all \( (\sim, \cO) \) of the
\( \SU(2) \)-sweep of \( \{[1:0]\} \times Q_{[\sim,\cO]}^{\circ} \).
In particular \( \tilde{\sigma} \) is injective on the dense
Zariski-open subset of \( \twist_Q = \PP^1 \times Q \) which is the \(
\SU(2) \)-sweep of \( \{[1:0]\} \times Q_{[\sim,\cO]}^{\circ} \) where
\( \sim \) and \( \cO \) are such that \( i \sim j \) if and only if
\( i=j \) and \( \cO = \{ 0 \} \).  Therefore the proof of
Theorem~\ref{thmsigmatilde} is also complete.

\begin{remark}
  \label{rem6.last}
  It follows from \cite{DKS2} Remark 3.4 and \S7 that the image of \(
  \sigma\colon Q \to \mathcal{R} \) is the closure in \( \mathcal{R}
  \) of the \( K_\C \)-sweep of the image \( \sigma(Q_T) \) in \(
  \mathcal{R} \) of the hypertoric variety \( M_T\hkq T_H \)
  associated to the hyperplane arrangement in \( \lie{t} \) given by
  the root planes (see Definition~\ref{defn3.10} above). Here \( \sigma(Q_T) =
  \sigma(\iota(M_T \hkq T_H)) \) where the composition \( \sigma \circ
  \iota\colon M_T \hkq T_H \to \mathcal{R} \) takes a quiver of the form
  \begin{equation*}
    \alpha_k = \begin{pmatrix}
      0 & \cdots &  0 & 0 \\
      \nu_1^k & 0  & \cdots & 0\\
      0 & \nu_2^k  & \cdots & 0\\
      & & \cdots  & \\
      0 & \cdots  & 0 & \nu_k^k
    \end{pmatrix} \end{equation*} and
  \begin{equation*}
    \beta_k = \begin{pmatrix}
      0 & \mu_1^k & 0  &  \cdots & 0\\
      0 & 0 & \mu_2^k  &  \cdots & 0\\
      & & \cdots &  &  \\
      0 & 0 & \cdots & 0  & \mu_k^k   \end{pmatrix}
  \end{equation*}
  to the section \( \rho_K + \rho_T + \sum_{j=1}^{n-1} \rho_j \) of
  \( \cO(2) \otimes (\lie k_\C \oplus \lie t_\C) \oplus
  \bigoplus_{j=1}^{n-1} \cO(\ell_j) \otimes \wedge^j\C^n \) where
  \begin{multline*}
    \rho_K(u,v) = \rho_T(u,v) = \left( (u\alpha_{n-1} + v
      \beta^*_{n-1})(-v\alpha^*_{n-1} + u \beta_{n-1}) \right)_0 \\
    =
    \begin{pmatrix}
      0 &   0 & \cdots   & 0 \\
      0 &
      \begin{aligned}
        &(u\nu_1^{n-1} + v \bar{\mu}_1^{n-1})\times{}\\
        &\quad(-v \bar{\nu}^{n-1}_1 + u
        \mu_1^{n-1})
      \end{aligned}
      & \cdots & 0\\
      &   &  \cdots  & \\
      0 & 0 & \cdots &
      \begin{aligned}
        &(u\nu_{n-1}^{n-1} + v \bar{\mu}_{n-1}^{n-1})\times{} \\
        &\quad(-v \bar{\nu}^{n-1}_{n-1} + u \mu_{n-1}^{n-1})
      \end{aligned}
    \end{pmatrix}\\
    - \sum_{i=1}^{n-1} \frac{(u\nu_i^{n-1} + v \bar{\mu}_i^{n-1})(-v
    \bar{\nu}^{n-1}_i + u \mu_i^{n-1})}{n} I_n
  \end{multline*}
  and
  \begin{equation*}
    \rho_j(u,v) = \prod_{k=j}^{n-1} \prod_{i=1}^j
    (u\nu_i^k + v \bar{\mu}_i^k) e_{j+1} \wedge \cdots \wedge e_n
  \end{equation*}
  where \( e_1, \ldots, e_n \) form the standard basis for \( \C^n \).

  For any \( p \in \PP^1 \) the projection to
  \begin{equation*}
    \bigoplus_{j=1}^{n-1} \cO(\ell_j)_p \otimes \wedge^j \C^n \cong
    \bigoplus_{j=1}^{n-1} \wedge^j \C^n 
  \end{equation*}
  of the evaluation \( \sigma_p \) of \( \sigma \) at \( p \) takes \( Q_T \) to the
  toric variety \( \overline{T_\C v} \) where \( v \in \bigoplus_{j=1}^{n-1}
  \wedge^j \C^n \) is the sum \( v = \sum_{j=1}^{n-1} v_j \) of highest
  weight vectors \( v_j \in \wedge^j \C^n \). Moreover if \( B = T_\C N \) is
  the standard Borel subgroup of \( K_\C = \SL(n,\C) \) and \( N=[B,B] \) is
  its unipotent radical which fixes the highest weight vectors \( v_j \),
  then this projection also takes the closure \( \overline{BQ_T} \) of the
  \( B \)-sweep of \( Q_T \) to \( \overline{Bv} = \overline{T_\C v} \). Thus this
  projection of \( \sigma_p \) takes \( Q= \overline{K_\C Q_T} = K
  \overline{BQ_T} \) to the universal symplectic implosion
  \begin{equation*}
    K \overline{T_\C v} = \overline{K_\C v} \subseteq
    \bigoplus_{j=1}^{n-1} \wedge^j \C^n. 
  \end{equation*}
  Furthermore \( \sigma_p \) takes \( Q_T \) (respectively \( \overline{BQ_T} \))
  birationally into the subvariety \( \lie{t}_\C \oplus \overline{T_\C
  v} \) (respectively \( \lie{n} \oplus \lie{t}_\C \oplus \overline{T_\C
  v} \cong \lie{n} \oplus \sigma_p(Q_T) \)) of
\begin{equation*}
\liek_\C \oplus \lie{t}_\C \oplus \bigoplus_{j=1}^{n-1} \wedge^j \C^n \cong 
\cO(2)_p \otimes (\liek_\C \oplus \lie{t}_\C) \oplus
\bigoplus_{j=1}^{n-1} \cO(\ell_j)_p \otimes \wedge^j \C^n,
\end{equation*}
where
\( \lie{n} \) is the Lie algebra of \( N \) and \( \lie{t}_\C \) is embedded
diagonally in \( \liek_\C \oplus \lie{t}_\C \), while \( \lie{n} \oplus
\lie{t}_\C \) is embedded via \( (\xi,\eta) \mapsto (\xi + \eta, \eta) \).

Since \( Q = \overline{K_\C Q_T} = K \overline{BQ_T} \) and
\( \overline{BQ_T} \) is \( T \)-invariant, we get a birational \( K \times
T \)-equivariant surjection
\begin{equation*}
K \times_T \overline{BQ_T} \to Q.
\end{equation*}
Similarly the twistor space \( \twist_Q \) of \( Q \) is the closure in \( \PP^1
\times \mathcal{R} \) of the \( K_\C \)-sweep of the twistor space
\( \twist_{Q_T} \) of the image \( Q_T \) in \( \mathcal{R} \) of the hypertoric
variety \( M_T\hkq T_H \). We have
\begin{equation*}
\twist_Q = \overline{K_\C \twist_{Q_T}} = K\overline{B \twist_{Q_T}}
\end{equation*}
and there is a birational surjection
\begin{equation*}
K \times_T \overline{B \twist_{Q_T}} \to \twist_Q.
\end{equation*}
Moreover \( \tilde{\sigma} \) restricts to \( T \)-equivariant birational
morphisms
\begin{equation*}
\tilde{\sigma}|_{\twist_{Q_T}} : \twist_{Q_T} \to \cO(2) \otimes \lie{t}_\C \oplus \overline{T_\C v}
\subseteq \cO(2) \otimes (\liek_\C \oplus \lie{t}_\C) \oplus
\bigoplus_{j=1}^{n-1} \cO(\ell_j) \otimes \wedge^j \C^n
\end{equation*}
and \( \tilde{\sigma}|_{\overline{B\twist_{Q_T}}} \colon
\overline{B\twist_{Q_T}} \to \cO(2) \otimes (\lie{n} \oplus
\lie{t}_\C) \oplus \overline{T_\C v}. \) Thus the twistor space \(
\twist_Q \) is birationally equivalent to \( K \times_T ((\cO(2)
\otimes \lie{n}) \oplus \twist_{Q_T}). \)
\end{remark}

\section{The twistor space of the universal hyperk\"ahler implosion
for \( \SU(n) \)}
\label{sec:twist-space-univ}

In this section we will describe the full structure of the twistor
space \( \twist_Q \) of \( Q \) in terms of the embedding \( \sigma \)
of \( Q \) in the space of holomorphic sections of the vector bundle
\( \cO(2) \otimes (\lie k_\C \oplus \lie t_\C) \oplus
\bigoplus_{j=1}^{n-1} \cO(\ell_j) \otimes \wedge^j\C^n \) over \( \PP^1 \),
and consider the cases when \( n=2 \) and \( n=3 \) in detail. The embedding
\( \sigma \) gives us an embedding \( \sigma_\twist \) of the twistor
space \( \twist_Q = \PP^1 \times Q \) of \( Q \) into \( \PP^1 \times
\mathcal{R} \), where as always
\begin{equation*}
  \mathcal{R} = 
  H^0(\PP^1,
  (\cO(2) \otimes (\lie k_\C \oplus  \lie t_\C)) \oplus 
  \bigoplus_{j=1}^{n-1}  \cO(\ell_j) \otimes \wedge^j\C^n).
\end{equation*}
This map \( \sigma_\twist \) is not holomorphic; however its
composition \( \tilde{\sigma} \) with the natural evaluation map from
\(  \PP^1 \times \mathcal{R} \) to
\begin{equation*}
  (\cO(2) \otimes (\lie k_\C \oplus  \lie t_\C)) \oplus 
  \bigoplus_{j=1}^{n-1}  \cO(\ell_j) \otimes \wedge^j\C^n
\end{equation*}
is holomorphic and \( K \times T \times \SU(2) \)-equivariant (see
Remark~\ref{remsigmatilde}), and \( \tilde{\sigma} \) is generically
injective by Theorem~\ref{thmsigmatilde}. Indeed as we saw at the end of the
last section, \( \tilde{\sigma} \) is injective on the dense subset of
\( \twist_Q = \PP^1 \times Q \) which is the union over all \( (\sim,
\cO) \) of the \( \SU(2) \)-sweep of \( \{[1:0]\} \times
Q_{[\sim,\cO]}^{\circ} \).

It follows that if \( \q \in Q \) lies in the stratum \( Q_{[\sim,
\cO]} \) indexed by equivalence relation \( \sim \) and nilpotent
coadjoint orbit \( \cO \), then we can find an open neighbourhood \(
U_{\q} \) of \( \q \) in \( Q_{[\sim, \cO]} \) and a closed subset \(
B_{\q} \) of \( \PP_1 \) of arbitrarily small area such that the
restriction of \( \tilde{\sigma} \) to the open subset \( (\PP^1
\setminus B_{\q})\times U_{\q} \) of \( \twist_Q \) is a holomorphic
embedding.  Since we can choose \( B_{\q} \) sufficiently small that
there is some \( s \in \SU(2) \) for which \( s(\PP^1 \setminus
B_{\q}) \) contains the points corresponding to the complex structures
\( i,j,k \) on \( Q \), it follows that the hypercomplex structure on
\( Q \) and thus the complex structure on \( \twist_Q \) are
determined by the embedding \( \sigma \).

Now consider the holomorphic section \( \omega_{[\sim,\cO]} \) of
\begin{equation*}
  \wedge^2 T^*_{F,Q_{[\sim,\cO]}} \otimes \cO(2)
\end{equation*}
where \( T^*_{F,Q_{[\sim,\cO]}} \) is the tangent bundle along the
fibres of the restriction of \( \pi\colon \twist_Q = \PP^1 \times Q
\to \PP^1 \) to \( \PP^1 \times Q_{[\sim,\cO]} \). The reasoning above
allows us to consider the restriction of \( \omega_{[\sim,\cO]} \) to
\( \{ [1:0] \} \times Q^\circ_{[\sim,\cO]} \).  By Theorem~\ref{thm6.8} this
can be identified with
\begin{equation*}
  K_\C \times_{(K_\sim)_\C} ((K_\sim)_\C \times_{R_{[\sim,\cO]}  }
  Q_{[\sim,\cO]}^{\circ,JCF})
\end{equation*}
where \( Q_{[\sim,\cO]}^{\circ,JCF} \) can in turn be identified with
an open subset of a hypertoric variety \( Q_{T,[\sim,\cO]} \) and
\(K_\C/(K_\sim)_\C\) and \( (K_\sim)_\C/R_{[\sim,\cO]} \) can be
identified with coadjoint orbits in \( \lie k_\C \) and \( (\lie
k_\sim)_\C \). The restriction of \( \omega_{[\sim,\cO]} \) is now
obtained from the Kirillov--Kostant construction as in
\S \ref {sec:nilpotent-cone}, combined with the holomorphic section of
\begin{equation*}
  \wedge^2 T^*_{F,Q_{T,[\sim,\cO]}} \otimes \cO(2)
\end{equation*}
associated to the twistor space \( \twist_{Q_{T,[\sim,\cO]}} \) of the
hypertoric variety \( Q_{T,[\sim,\cO]} \).

Finally, as in Remark~\ref{nilkhnought}, the real structure on the twistor
space \( \twist_Q \) is determined by the embedding \( \sigma_\twist
\) and the real structure on \(  \PP^1 \times \mathcal{R} \) determined by
the real structure \( \zeta \mapsto -1/\bar{\zeta} \) on \( \PP^1 \),
together with the real structures \( \eta \mapsto -\bar{\eta}^T \) on
\( \lie k_\C \) and \( \lie t_\C \) and the real structures on \(
\wedge^j \C^n \) induced by the standard real structure on \( \C^n \).

\begin{example}
  Consider the case when \( n=2 \) and \( K=\SU(2) \). Then \( M \) is the space
  of quivers of the form
  \begin{equation}
    \label{quiver2}
    \C \stackrel[\beta]{\alpha}{\rightleftarrows}
    \C^2 
  \end{equation}
  and \( Q = M\hkq \SU(1) = M \cong \C^2 \oplus (\C^2)^* \cong
  \mathbb{H}^2 \) (see \cite{DKS} Example 8.5). There are two
  equivalence relations on \( \{1,2\} \); let \( \sim_1 \) denote the
  equivalence relation with one equivalence class \( \{1,2\} \) and let
  \( \sim_2 \) denote the equivalence relation with two equivalence
  classes \( \{1\} \) and \( \{2\} \). Then \( (\liek_{\sim_1})_\C \) has two
  nilpotent orbits, \( \cO_{1,0} = \{ 0 \} \) and the orbit \( \cO_{1,1} \) of
\begin{equation*}
\begin{pmatrix} 0 & 1 \\ 0 & 0
\end{pmatrix},
\end{equation*}
while the only nilpotent orbit in
\( (\liek_{\sim_2})_\C \) is \( \cO_{2,0} = \{ 0 \} \). The corresponding
stratification of \( Q = \mathbb{H}^2 \) is
\begin{equation*}
Q = Q_{[\sim_1,\cO_{1,0}]} \sqcup Q_{[\sim_1,\cO_{1,1}]} \sqcup Q_{[\sim_2,\cO_{2,0}]}
\end{equation*}
where
\begin{equation*}
Q_{[\sim_1,\cO_{1,0}]}=  Q_{[\sim_1,\cO_{1,0}]}^\circ = \{ (0,0)\}.
\end{equation*}
Moreover \( Q_{[\sim_1,\cO_{1,1}]}= Q_{[\sim_1,\cO_{1,1}]}^\circ  \)
consists of nonzero quivers of the form Eq.~\eqref{quiver2} satisfying the
hyperk\"aher moment map equations for the action of \( \Un(1) \), so by
Example 2.7 they belong to the \( K \)-sweep of the set of quivers
satisfying
\begin{equation*}
\alpha = \begin{pmatrix} \sqrt{d} \\ 0 \end{pmatrix}, \,\,\, \beta = (0 \,\, \sqrt{d})
\end{equation*}
for real and strictly positive \( d \). The quotient of \( Q_{\sim_1} =
Q_{[\sim_1,\cO_{1,0}]} \sqcup Q_{[\sim_1,\cO_{1,1}]} \) by the action of
\( \Un(1) \) is the nilpotent cone in the Lie algebra \( \liek_\C \) of
\( \SL(2,\C) \), with the quotient map given by \( (\alpha,\beta) \mapsto
\alpha \beta \).

Finally \( Q_{[\sim_2,\cO_{2,0}]}^\circ \) consists of the quivers
Eq.~\eqref{quiver2} such that the \( 2 \times 2 \) matrix \( \alpha
\beta \) has distinct eigenvalues, while its sweep
\begin{equation*}
Q_{[\sim_2,\cO_{2,0}]} = \SU(2) Q_{[\sim_2,\cO_{2,0}]}^\circ
\end{equation*}
by the action of \( \SU(2) \) which rotates the complex structures on \( Q \)
consists of the quivers Eq.~\eqref{quiver2} such that the \( 2 \times 2 \)
matrix
\begin{equation*}
(u \alpha + v \beta^*) (-v \alpha^* + u \beta)
\end{equation*}
has distinct eigenvalues for some (and hence generic) choice of \( (u,v)
\in \C^2 \).

Now consider the map
\begin{equation*}
  \sigma\colon Q = \mathbb{H}^2 \to \mathcal{R} = H^0(\PP^1, \cO(2)
  \otimes (\liek_\C \oplus \lie{t}_\C) \oplus \cO(1) \otimes \C^2).
\end{equation*}
The projection of \( \sigma \) onto
\begin{equation*}
H^0(\PP^1, \cO(1) \otimes \C^2) \cong H^0(\PP^1, \cO(1)) \otimes \C^2 \cong \C^2 \otimes \C^2
\end{equation*}
takes a quiver Eq.~\eqref{quiver2} to the section of \( \cO(1) \otimes
\C^2 \) over \( \PP^1 \) given by
\begin{equation*}
(u,v) \mapsto u\alpha + v \beta^*,
\end{equation*}
and hence is bijective. Moreover the projection onto \( H^0(\PP^1,
\cO(2) \otimes (\liek_\C \oplus \lie{t}_\C)) \) is the map induced by
the hyperk\"ahler moment maps for the \( K\times T \)-action on \( Q \), and
so \( \sigma \) embeds \( Q \) into \( \mathcal{R} \) as the graph of this map.

Since \( Q= \mathbb{H}^2 \) is a flat hyperk\"ahler manifold its twistor
space is the vector bundle
\begin{equation*}
\twist_Q = \cO \otimes (\C^2 \oplus (\C^2)^*)
\end{equation*}
over \( \PP^1 \). We can cover \( \PP^1 \) with two open subsets \( \{\zeta \in
\PP^1: \zeta \neq \infty\} \) and \( \{\zeta \in \PP^1: \zeta \neq 0\} \),
and thus cover \( \twist_Q \) with two coordinate patches where \( \zeta
\neq \infty \) and where \( \zeta \neq 0 \) with coordinates
\begin{equation*}
(\alpha, \beta, \zeta) \mbox{ and } (\tilde{\alpha}, \tilde{\beta}, \tilde{\zeta})
\end{equation*}
for \( \alpha,\tilde{\alpha} \in \C^2  \) and \( \beta, \tilde{\beta} \in
(\C^2)^* \) related by
\begin{equation*}
\tilde{\zeta} = 1/\zeta, \,\, \tilde{\alpha} = \alpha /\zeta, \,\, \tilde{\beta} = \beta /\zeta;
\end{equation*}
we have similar coordinates on \( \cO(2) \otimes (\liek_\C \oplus
\lie{t}_\C) \oplus \cO \otimes \C^2 \).  With respect to these
coordinates the map \( \tilde{\sigma}\colon \twist_Q \to \cO(2)
\otimes (\liek_\C \oplus \lie{t}_\C) \oplus \cO(1) \otimes \C^2 \) is
given by
\begin{equation*}
  \tilde{\sigma}(\alpha,\beta,\zeta) = (\alpha \beta -
  \frac{\mathrm{tr}(\alpha \beta)}{2} I_2, \beta \alpha, \alpha,
  \zeta) 
\end{equation*}
(see \S \ref {sec:towards-an-embedding}). Observe that where the
coordinate \( \alpha \) is nonzero (or equivalently defines an
injective linear map \( \alpha\colon \C \to \C^2 \)) then we can
recover \( \alpha, \beta \) and \( \zeta \) from \(
\tilde{\sigma}(\alpha,\beta,\zeta) \), but that \(
\tilde{\sigma}(0,\beta,\zeta) = (0,0,0,\zeta) \) for any \( \beta
\). Thus \( \tilde{\sigma} \) is only generically injective on the
twistor space \( \twist_Q \).

Recall from \S \ref {sec:twistor-spaces} that the real structure on \(
\twist_Q \) is given in these coordinates by
\begin{equation*}
  (\alpha, \beta, \zeta) \mapsto (\bar{\beta}/\bar{\zeta}, -
  \bar{\alpha}/\bar{\zeta}, -1/\bar{\zeta}); 
\end{equation*}
this is induced via the embedding \( \sigma_\twist \) from the real
structure on \( \PP^1 \times \mathcal{R} \) determined by the standard
real structures on \( \PP^1 \), \( \liek_\C \), \( \lie{t}_\C \) and \( \C^2 \).

Lastly observe that \( Q_{[\sim_1,\cO_{1,0}]}=
Q_{[\sim_1,\cO_{1,0}]}^\circ = \{ (0,0)\} \) and that
\( Q_{[\sim_1,\cO_{1,1}]}= Q_{[\sim_1,\cO_{1,1}]}^\circ  \) can be
identified with the regular nilpotent orbit in \( \liek_\C \), and the
holomorphic symplectic form on each is obtained by the
Kirillov-Kostant construction. Moreover
\begin{equation*}
Q_{[\sim_2,\cO_{2,0}]}^\circ \cong K_\C \times _{T_\C} Q_{[\sim_2,\cO_{2,0}]}^{\circ,JCF}
\end{equation*}
where \( Q_{[\sim_2,\cO_{2,0}]}^{\circ,JCF} \) can be identified with an
open subset of \( \mathbb{H} \); the holomorphic symplectic form on this
is obtained from the Kirillov-Kostant construction on the adjoint
orbit \( K_\C/T_\C \) combined with the flat structure on \( \mathbb{H} \).
\end{example}

\begin{example}
  Finally let us consider briefly the case when \( n=3 \) and
  \( K=\SU(3) \). We have five equivalence relations \( \sim_{123},
  \sim_{12,3}, \sim_{13,2}, \sim_{23,1} \) and \( \sim_{1,2,3} \) on
  \( \{1,2,3\} \) given by the partitions
\begin{equation*}
\{ \{1,2,3\}\}, \,\, \{ \{1,2\}, \{3\}\}, \,\, \{ \{1,3\}, \{2\} \}, \,\, \{ \{ 2,3\}, \{1\} \}, \,\,
\{ \{1\}, \{2\}, \{3\} \}.
\end{equation*}
The Lie algebra of \( (K_{\sim_{123}})_\C =
K_\C \) has three nilpotent orbits \( \cO_{123,j} \) for \( j=1,2,3 \), and the
corresponding strata \( Q_{[\sim_{123},\cO_{123,j}]} \) can be described
using Example 2.8. At the other extreme the Lie algebra of
\( (K_{\sim_{1,2,3}})_\C = T_\C \) has only the zero nilpotent orbit, and
the structure of the stratum \( Q_{[\sim_{1,2,3},\{0\}]} \) is similar to
that of \( Q_{[\sim_2,\cO_{2,0}]} \) in Example 7.1.  In between the Lie
algebras of
\begin{equation*}
(K_{\sim_{12,3}})_\C \cong (K_{\sim_{13,2}})_\C \cong (K_{\sim_{23,1}})_\C \cong \GL(2,\C)
\end{equation*}
have two nilpotent orbits each, the zero orbit and the regular
nilpotent orbit. These give us the six remaining strata
\( Q_{[\sim,\cO]} = \SU(2)Q_{[\sim,\cO]}^\circ  \), for each of which the
open subset \( Q_{[\sim,\cO]}^\circ \) of \( Q_{[\sim,\cO]} \) has the form
\begin{equation*}
Q_{[\sim,\cO]}^\circ \cong K_\C \times_{\GL(2,\C)} \left( \GL(2,\C) \times_{R_{[\sim,\cO]}} 
  Q_{[\sim,\cO]}^{\circ,JCF} \right)
\end{equation*}
where \( R_{[\sim,\cO]} \) is the
stabiliser of \( \cO \) in \( \GL(2,\C) \) and \( Q_{[\sim,\cO]}^{\circ,JCF} \)
can be identified with an open subset of \( \mathbb{H} \).
\end{example}

\section{More general compact Lie groups}
\label{sec:general-compact-K}
Our future aim and the main motivation for this paper is to be able to
construct the hyperk\"ahler implosion of a hyperk\"ahler manifold \( M \)
with a Hamiltonian action of any compact Lie group \( K \). For this it
suffices to construct a universal hyperk\"ahler implosion
\( (T^*K_\C)_{\mathrm{hkimpl}} \) of the hyperk\"ahler manifold \( T^*K_\C \)
(see \cite{Kronheimer:cotangent}) with suitable properties. In
particular \( (T^*K_\C)_{\mathrm{hkimpl}} \) should be a stratified
hyperk\"ahler space with a Hamiltonian action of \( K \times T \) where
\( T \) is a maximal torus of \( K \); then we can define the hyperk\"ahler
implosion \( M_{\mathrm{hkimpl}} \) as the hyperk\"ahler quotient of \( M
\times (T^*K_\C)_{\mathrm{hkimpl}} \) by the diagonal action of \( K \).

As Guillemin, Jeffrey and Sjamaar observed in \cite{GJS} for
symplectic implosion, it suffices to consider the case when \( K \) is
semisimple, connected and simply connected.  In this case, as was
noted in \S \ref {sec:symplectic-implosion} above, we can embed the
universal symplectic implosion \( (T^*K)_{\mathrm{impl}} \) in the
complex affine space
\begin{equation*}
E = \bigoplus_{\varpi \in \Pi} V_\varpi
\end{equation*}
where \( \Pi \) is a minimal generating set for the monoid of dominant
weights; here if \( \varpi \in \Pi \) then \( V_\varpi \) is the \( K
\)-module with highest weight \( \varpi \) and \( T \) acts on \(
V_\varpi \) as multiplication by this highest weight. As we recalled
in \S \ref {sec:symplectic-implosion}, \( (T^*K)_{\mathrm{impl}} \) is
embedded in \( E \) as the closure of the \( K_\C \)-orbit \( K_\C v
\) where \( v \) is the sum \( \sum_{\varpi \in \Pi} v_\varpi \) of
highest weight vectors \( v_\varpi \in V_\varpi \), or equivalently \(
(T^*K)_{\mathrm{impl}} = K (\overline{T_\C v}) \) where \(
\overline{T_\C v} \) is the toric variety associated to the positive
Weyl chamber \( \lie{t}_+ \).

This representation \( E = \bigoplus_{\varpi \in \Pi} V_\varpi \) of
\( K \) gives us an identification of \( K \) with a subgroup of \(
\prod_{\varpi \in \Pi} \SU(V_\varpi) \).  Theorem~\ref{thmsigma} gives us an
embedding \( \sigma \) of the universal hyperk\"ahler implosion for \(
\prod_{\varpi \in \Pi} K_\varpi \), where \( K_\varpi = \SU(V_\varpi)
\) has maximal torus \( T_\varpi \), in
\begin{equation*}
\prod_{\varpi \in \Pi} \mathcal{R}_\varpi = H^0(\PP^1, \bigoplus_{\varpi \in \Pi} \cO(2) \otimes (\liek_\varpi 
\oplus \lie{t}_\varpi)_\C \oplus \bigoplus_{j=1}^{\dim V_\varpi -1} \cO(\ell_j) \otimes \wedge^j V_\varpi).
\end{equation*}
We may assume that the inclusion of \( K \) in \( \prod_{\varpi \in \Pi} K_\varpi  \) restricts to an inclusion of its maximal torus \( T \) in the maximal torus \( \prod_{\varpi \in \Pi} T_\varpi  \) of \( \prod_{\varpi \in \Pi} K_\varpi  \)
as \( K \cap \prod_{\varpi \in \Pi} T_\varpi \), and that the intersection \( \liek \cap \prod_{\varpi \in \Pi} (\lie{t}_\varpi)_+  \) in \( \prod_{\varpi \in \Pi} \liek_\varpi  \) is a positive Weyl chamber \( \lie{t}_+ \) for \( K \).
Then the hypertoric variety for \( T \) associated to the hyperplane arrangement given by the root planes
in \( \lie{t} \) embeds  in the hypertoric variety for \( \prod_{\varpi \in \Pi} T_\varpi  \) associated to the
hyperplane arrangement given by the root planes in \( \prod_{\varpi \in \Pi} \lie{t}_\varpi  \), and
thus maps into \( \prod_{\varpi \in \Pi} \mathcal{R}_\varpi \). 

By analogy with the situation
described in \cite{GJS} for symplectic implosion and using Remark~\ref{rem6.last}, we expect the twistor 
space  \( \twist_{(T^*K_\C)_{\mathrm{hkimpl}}} \) for the universal hyperk\"ahler implosion \( (T^*K_\C)_{\mathrm{hkimpl}} \)
to embed in the intersection in \( \PP^1 \times  \prod_{\varpi \in \Pi} \mathcal{R}_\varpi \) of the corresponding twistor space for \(  (T^* (\prod_{\varpi \in \Pi} K_\varpi)_\C)_{\mathrm{hkimpl}} \) and 
\begin{equation*}
\PP^1 \times  H^0(\PP^1,  \cO(2) \otimes (\liek_\C 
\oplus \lie{t}_\C) \oplus \bigoplus_{\varpi \in \Pi} \bigoplus_{j=1}^{\dim V_\varpi -1} \cO(\ell_j) \otimes \wedge^j V_\varpi).
\end{equation*}
Moreover we expect the image of this embedding to be the closure of
the \( K_\C \)-sweep of the image of the twistor space of the
hypertoric variety for \( T \) associated to the hyperplane
arrangement given by the root planes in \( \lie{t} \), and that the
full structure of the twistor space \(
\twist_{(T^*K_\C)_{\mathrm{hkimpl}}} \) for the universal
hyperk\"ahler implosion is obtained from this embedding as in
\S \ref {sec:twist-space-univ} above.

\end{document}